\documentclass[12pt, a4]{amsart}
\usepackage{amsgen,amsmath,amstext,amsbsy,amsopn,amsfonts,amssymb}
\usepackage{amsmath,amssymb,amsfonts,amsthm,enumerate}
\usepackage{mathrsfs}
\usepackage{amsthm}
\usepackage[alphabetic]{amsrefs}

\newtheorem{theorem}{Theorem}[section]
\newtheorem{lemma}[theorem]{Lemma}
\newtheorem{corollary}[theorem]{Corollary}
\newtheorem{proposition}[theorem]{Proposition}
\newtheorem{obs}[theorem]{Observation} \newtheorem{defi}[theorem]{Definition}

\newenvironment{definition}{\begin{defi}\rm}{\end{defi}}
\newtheorem{exa}[theorem]{Example}

\newtheorem{rem}[theorem]{Remark}
\newenvironment{remark}{\begin{rem}\rm}{\end{rem}}
\newtheorem{rems}[theorem]{Remarks}

\newtheorem{ack}[theorem]{Acknowlegment}

\def\n{\noindent}
\def\H{\mathcal H}
\def\L{\mathcal L}
\def\K{\mathcal K}
\def\P{\mathcal P}
\def\M{\mathcal M}

\def\NN{{\mathbf N}}
\def\ZZ{{\mathbf Z}}
\def\CC{{\mathbf C}}
\def\RRR{{\mathbf R}}
\def\QQ{\mathbf Q}

\def\AA{{\mathbf A}}
\def\RR+{{\mathbf R}^*}
\def\TT{\mathbf T}

\def\LL{\mathbf L}
\def\HH{\mathbf H}
\def\kk{\mathbf k}
\def\GG{\mathbf G}

\def\Un{\mathbf 1}

\def\Q_p{{\mathbf Q}_p}
\def\SS{\mathbf S}

\def\UU{\mathbf U}

\def\eps{\varepsilon}
\def\Ga{\Gamma}
\def\ga{\gamma}
\def\g{\gamma}
\def\La{\Lambda}
\def\la{\lambda}
\def\vfi{\varphi}

\def\tous{\qquad\text{for all}\quad}

\def\Aut{{\mathrm Aut}}
\def\Aff{{\mathop{\rm Aff}}}

\def\Ind{{\mathrm Ind}}

\def\Ad{{\mathrm Ad}}

\def\Tr{{\mathrm Tr}}
\def\Char{{\mathrm Char}}

\def\Prob{{\mathrm Prob}}

\def\tout{\quad\text{for all}\quad}
\def\et{\qquad\text{and}\qquad}
\def\Un{\mathbf 1}
\begin{document}

\date{\today}
\title[Characters of algebraic groups]{Characters of algebraic groups over number fields}

\address{Bachir Bekka \\ Univ Rennes \\ CNRS, IRMAR--UMR 6625\\
Campus Beaulieu\\ F-35042  Rennes Cedex\\
 France}
\email{bachir.bekka@univ-rennes1.fr}
\author{Bachir Bekka and Camille Francini}
\address{Camille Francini \\ Univ Rennes \\ CNRS, IRMAR--UMR 6625\\
Campus Beaulieu\\ F-35042  Rennes Cedex\\
 France}
\email{camille.francini@ens-rennes.fr}

\thanks{The authors acknowledge the support  by the ANR (French Agence Nationale de la Recherche)
through the projects Labex Lebesgue (ANR-11-LABX-0020-01) and GAMME (ANR-14-CE25-0004)}
\begin{abstract}
Let  $k$ be a number field, $\GG$ an algebraic group 
defined over $k$, and $\GG(k)$ the group of  $k$-rational points  in $\GG.$
We determine  the set  of functions on $\GG(k)$ which are of positive type and conjugation invariant,  under the assumption that $\GG(k)$ is generated by its unipotent elements.
An essential step in the proof is the classification of the $\GG(k)$-invariant 
ergodic probability measures on an adelic solenoid naturally associated to $\GG(k);$
this last result is deduced from  Ratner's measure rigidity theorem for  
 homogeneous spaces of $S$-adic Lie groups.
 \end{abstract}

\maketitle
\section{Introduction}
\label{S0}

Let  $k$ be a field and $\GG$ an algebraic group  defined over $k$. When $k$ is a local field (that is,  a non discrete locally compact  field),  the group $G=\GG(k)$ of  $k$-rational points  in $\GG$ is a locally compact group for the topology induced by $k$.
In this case (and when moreover $k$ is of characteristic zero),
much is known (\cite{Harish}, \cite{Duflo}) about the unitary dual $\widehat{G}$ of $G$, the set of equivalence classes of irreducible unitary representations of $G$ in Hilbert spaces. 
By way of contrast,  if $k$ is a global field  (that is, either  a number field or a function field in one variable over a finite field), then $G$  is a countable infinite group 
and, unless $G$ is abelian,   the  classification of  $\widehat{G}$ is a hopeless task, as follows from work of Glimm and Thoma (\cite{Glimm},\cite{Thoma1}). 
In this case, a sensible substitute for  $\widehat{G}$ is the set of characters of $G$ we are going to define.

Let $G$ be  a  group. Recall that a function $\vfi:G\to \CC$
is of positive type if the complex-valued matrix $(\vfi(g_j^{-1} g_i))_{1\leq i,j\leq n}$ is positive semi-definite for any $g_1, \dots, g_n$ in $G.$

A function of positive type $\vfi$ on $G$ which is central (that is, constant
on conjugacy classes) and normalized (that is, $\vfi(e)=1$) will be called a \textbf{trace}
on $G$. 
The set $\Tr(G)$ of traces on $G$ is  a 
convex subset of the unit ball of $\ell^\infty (G)$ which is 
compact in the topology of pointwise convergence.
The extreme points $\Tr(G)$  are called the \textbf{characters} of $G$ and the set 
they constitute will be denoted by $\Char(G).$

Besides providing an alternative dual space of a group $G,$ characters and traces appear in various situations. Traces of $G$ are  tightly connected to representations of  $G$ in  the unitary group of  tracial von Neumann algebras (see below and Subsection~\ref{TracesAndVNalgebras}). 
The space $\Tr(G)$ of traces on $G$ encompasses the  lattice
of all normal subgroups of $G,$ since the characteristic function of every normal
subgroup is a trace on $G.$
More generally, every measure preserving action of $G$ on 
a probability space gives rise to an invariant random subgroup (IRS) on $G$ and therefore to a  trace on $G$ (see \cite[\S 9]{Gelander}).

The study of characters on infinite discrete groups was initiated by Thoma (\cite{Thoma1},\cite{Thoma2}) and the space $\Char(G)$ was determined for various  groups
$G$ (see \cite{Thoma3}, \cite{Kirillov}, \cite{Ovcinnikov},
 \cite{Skudlarek}, \cite{CareyEtAl1}, \cite{Bekka-Inv}, \cite{Dudko-Medynets}, \cite{Peterson},  \cite{Peterson-Thom}, \cite{Boutonnet-Houdayer}).

Observe that  our traces are often called characters in the literature
(see for instance \cite{Dudko-Medynets}, \cite{Peterson-Thom}).

\vskip.2cm
Let $k$ be a number field (that is, a finite extension of $\QQ$)
and  $\GG$ a connected linear algebraic group defined over $k.$
In this paper, we will give a complete description 
of $\Char(G)$ for $G=\GG(k)$, under the assumption that 
$G$ is generated by its unipotent one-parameter subgroups.
A  \textbf{unipotent one-parameter subgroup} of   $G$  is a subgroup  of the form 
$\{u(t)\mid t\in k\},$ where 
$u: \GG_a\to \GG$ is a nontrivial $k$-rational homomorphism from the
additive group $\GG_a$ of dimension 1 to $\GG.$

The case  where $\GG$ is quasi-simple over $k$ was treated in \cite{Bekka}
and the result is that 
$$
\Char(G)=\{\widetilde{\chi}\mid \chi \in \widehat{Z}\} \cup \{1_G\}
$$
where $Z$ is the (finite) center of $G$ and 
$\widetilde{\chi}:G\to \CC$ is defined by $\widetilde{\chi}=\chi$ on $Z$
and $\widetilde{\chi}=0$ on $G\setminus Z.$
When $\GG$ is semi-simple,  the computation
$\Char(G)$ can easily be reduced to  the quasi-simple case
(see \cite[Proposition 5.1]{Bekka}; see also Corollary~\ref{Cor-CharProduct} below).
\vskip.2cm

We now turn to a general  connected linear algebraic group $\GG$ over $k.$
The  unipotent radical  $\UU$ of $\GG$ is defined over $k$ and 
there exists a connected reductive $k$-subgroup $\LL$, called a Levi subgroup,  
 such that  $\GG=\LL \UU$ (see \cite{Mostow}).
 Set $U:= \UU\cap G$ and $L:=\LL\cap G$. Then, we have a corresponding 
semi-direct  decomposition  $G=LU$, called
 the \textbf{Levi decomposition}\,  of $G$ (see \cite[Lemma 2.2]{Lipsman}).
 
  Recall that $\LL=\TT \LL'$ is an almost direct product (see Subsection~\ref{SS: ComputationSS} for this notion)
of a central $k$-torus $\TT$ and the  derived subgroup $\LL'$,
which is a semi-simple $k$-group  Assume that $G$ is generated by its unipotent one-parameter subgroups.
Then the same holds for $L$. Since every unipotent one-parameter subgroup
of $\LL$ is contained in $ \LL',$ it follows that $G=  \LL'(k) U,$
that is, the Levi subgroup $\LL$ is semi-simple.

 We will describe  $\Char(G)$ in terms of data attached to $L$ and 
 the  action of $L$ on the Lie algebra  $\mathrm{Lie}(\UU)$  of $U.$

 \vskip.2cm
 The set $\mathfrak{u}$  of $k$-points of  $\mathrm{Lie}(\UU)$  is a Lie algebra over $k$ and the exponential map $\exp: \mathfrak{u}\to U$ is a bijective map.
 For every  $g$ in $G,$ the automorphism of $U$ given by conjugation with 
$g$ induces an automorphism $\Ad (g)$ of the Lie algebra $\mathfrak{u}$
(see Subsection~\ref{SS:TracesNipotentGroups}).

Let   $\widehat{\mathfrak{u}}$ be the Pontrjagin dual of  $\mathfrak{u}$, that is, 
 the group of unitary character of the  additive group of $\mathfrak{u}$.
  We associate to every $\lambda\in \widehat{\mathfrak{u}}$ the following subsets 
$\mathfrak{k}_\lambda, \mathfrak{p}_\lambda$ of $\mathfrak{u}$ and  $L_\la$ of $L$:
\begin{itemize}
\item 
$\mathfrak{k}_\lambda$ is the set of elements $X\in \mathfrak{u}$ such that 
$$
\lambda(\Ad(g)(tX))=1 \tout g\in G, t\in k;
$$
\item $\mathfrak{p}_\lambda$ is the set of elements $X\in \mathfrak{u}$ such that 
$$
\lambda(\Ad(g)(tX))=\lambda(tX) \tout g\in G, t\in k;
$$
\item $L_\la$ is the set  of $g\in L$ such that $\Ad(g)(X)\in X+\mathfrak{k}_\lambda$
for every $X\in  \mathfrak{u}.$
\end{itemize}
Then $\mathfrak{k}_\lambda$ and $\mathfrak{p}_\lambda$ are 
$L$-invariant ideals  of  $\mathfrak{u}$ and $L_\la$ is the kernel of the quotient representation of $L$ on $\mathfrak{u}/\mathfrak{k}_\lambda$. 
 
The sets $K_\la:=\exp(\mathfrak{k}_\lambda)$ and $P_\la:=\exp(\mathfrak{p}_\lambda)$ are $L$-invariant  normal  subgroups of $U$.
Moreover (see Proposition~\ref{Pro-KCenter}),
$P_\lambda$ is the inverse image under 
the canonical projection $U\to U/K_\lambda$
of the elements in $U/K_\lambda$ contained in the center of $G/K_\lambda$
and  
$$\chi_\la: P_\la\to \SS^1, \qquad \exp(X)\to \lambda(X)$$
 is a $G$-invariant unitary character of $P_\la$, which is trivial
 on $K_\la$. 

Let  $\Ad^*$ denote the coadjoint action (that is, the dual action) of $G$ on $\widehat{\mathfrak{u}}.$ 
 We say that  $\la_1, \la_2\in \widehat{\mathfrak{u}}$ have the same \textbf{quasi-orbit} under $G$ if  the  closures  of $\Ad^*(G) \la_1$ and $\Ad^*(G) \la_2$   in the compact group $\widehat{\mathfrak{u}}$  coincide.
 
 We can now state our main result.

 \begin{theorem}
 \label{Theo-GenAlgGroup}
 Let $G=\GG(k)$ be the group of  $k$-rational points of a connected linear algebraic group $\GG$ over a number field $k$. Assume that $G$ is generated by its unipotent one-parameter subgroups and let $G=LU$ be a Levi decomposition of $G$.
For $ \la\in \widehat{\mathfrak{u}}$ and  $\vfi\in \Char(L_\la),$ define $\Phi_{(\la, \vfi)}: G\to \CC$ by  
$$\Phi_{(\la, \vfi)} (g)=\begin{cases}
\vfi(g_1)\chi_\la(u)&\text{if } g=g_1 u \ \text{for}\  g_1\in L_\la,  u\in P_\lambda \\
0&\text{otherwise}.
\end{cases}
$$
\begin{itemize}
\item[(i)] We have
$$\Char(G)=\left\{\Phi_{(\la, \vfi)} \mid\la\in \widehat{\mathfrak{u}}, \vfi\in \Char(L_\la)\right\}.$$
\item[(ii)] Let $\la_1, \la_2\in \widehat{\mathfrak{u}}$ and $ \vfi_1\in \Char(L_{\la_1}),
 \vfi_2\in \Char(L_{\la_2}).$ Then $\Phi_{(\la_1, \vfi_1)}= \Phi_{(\la_2, \vfi_2)}$
 if and only if $\la_1$ and $\la_2$ have the same quasi-orbit under the coadjoint action $\Ad^*$ and $\vfi_1=\vfi_2.$
 \end{itemize}
 \end{theorem}
 
 A few words about the proof of Theorem~\ref{Theo-GenAlgGroup}  are in order.
 The essential step  consists in the analysis of the restriction $\vfi|_U$ to $U$ of a given character $\vfi\in \Char(G).$
 A first crucial fact is that $\psi=\vfi|_U\circ \exp$ is a $G$-invariant function of positive type on  $\mathfrak{u}$ (for the underlying  abelian group structure) and
 is extremal under such  functions (see Proposition~\ref{Pro-TracesNilpotent} and Theorem~\ref{Prop-GNS-InnerAut}); the Fourier transform
 of $\psi$ is a $G$-invariant ergodic probability measure on $\widehat{\mathfrak{u}},$
 which is an adelic solenoid.
We classify  all such measures  (Theorem~\ref{Theo-ErgodicityAdelic}); 
for this,   we use Ratner's measure rigidity results
  for   homogeneous spaces of $S$-adic Lie groups (see \cite{Ratner2}, \cite{Margulis-Tomanov}). A corresponding description,  based  on Ratner's topological rigidity results, is given for the $G$-orbit closures in $\widehat{\mathfrak{u}}$ (Theorem~\ref{Theo-OrbitClosure-Adelic}).
 
 \begin{remark}
\label{Rem-Theo1}
(i) For every $\la\in  \widehat{\mathfrak{u}}$, the group $L_\la$ as defined above is the set of $k$-points of a normal 
subgroup $\LL_\la$ of $\LL$ defined over $k$; indeed, $\LL_\la$ 
is the  the kernel of
the $k$-rational representation of $\LL$ on the $k$-vector space
$\mathfrak{u}/\mathfrak{k}_\la.$ (Observe that $\LL_\la$ may be non connected.)
The set $\Char(L_\la)$ can easily be  described by the results in \cite{Bekka}  mentioned above (see   Proposition~\ref{Prop-CharProduct2} below).

\n
(ii) Theorem~\ref{Theo-GenAlgGroup}
allows a full classification of $\Char(G)$ for any 
group $G$ as above, through the following procedure:
\begin{itemize}
\item determine the $L$-invariant ideals of $\mathfrak{u};$
\item fix an  $L$-invariant ideal  $\mathfrak{k}$ of $\mathfrak{u};$  determine 
 the space $\overline{\mathfrak{p}}$ of $L$-fixed  elements in the center of $\mathfrak{u}/\mathfrak{k}$ and let
$ \mathfrak{p}$ be its the inverse image in $\mathfrak{u}$;
\item determine the subgroup $L(\mathfrak{k}, \mathfrak{p})$ of $L$
of all elements which act trivially on $\mathfrak{p}/\mathfrak{k};$
determine $\Char(L(\mathfrak{k}, \mathfrak{p}));$
\item let  $\la\in  \widehat{\mathfrak{u}}$ with  
$\mathfrak{k}_\lambda=\mathfrak{k}$; then $\mathfrak{p}_\la=\mathfrak{p}$
and for $\vfi\in \Char (L(\mathfrak{k}, \mathfrak{p}))$, write  $\Phi_{(\la, \vfi)}\in\Char(G).$ 
\end{itemize}
See Section~\ref{S:Examples} for some examples.

\n
(iii) The assumption that $G$ is generated by its unipotent one-parameter subgroups
is equivalent to the assumption that the Levi component $L$
of $G$ is semi-simple and that $L^+=L$, where $L^+$ is the subgroup of $L$ defined as in
\cite[\S 6]{Borel-Tits2}. A necessary condition for 
the equality $L^+=L$ to hold is that every non-trivial simple  
 algebraic normal subgroup of $\LL$ is $k$-isotropic (that is,  $k-\rm{rank}(\LL) \geq 1$).
 It is known that $L^+=L$ when $\LL$ is simply-connected  and split or quasi-split  over $k$ (see \cite[Lemma 64]{Steinberg}).
 
 \n
 (iv)  A general result about $\Char(G)$ cannot be expected  when the condition
 $L=L^+$ is dropped; indeed,  not even the normal subgroup structure of $\Char(L)$ is  known is general when $\LL$ is $k$-anisotopic (see \cite[Chap. 9]{Platonov-Rapinchuk})
 
 We do not know whether an appropriate version of Theorem~\ref{Theo-GenAlgGroup} 
 is valid when $k$ is of positive characteristic (say, when  
 $k=F(X)$ for a finite field $F$).  A first obstacle
 to overcome is that a Levi subgroup   of $\GG$ does not necessarily exist;
 a second one is the less tight relationship between unipotent groups
 and their Lie algebras; finally,   Ratner's measure rigidity theorem is not known in full generality  (see \cite{Einsiedler-Ghosh} for a partial result).  
 
\n
(v) In the case where $\GG$ is unipotent, that is, $G=U,$ 
we obtain a ``Kirillov type" description of  $\Char(U)$:
the map  $\Phi: \widehat{\mathfrak{u}} \to \Char(U),$ defined by  $\Phi(\la) (u)=\chi_\la(u)$ for  $u\in P_\la$
and $\Phi(\la)(u)=0$ otherwise, factorizes to a bijection
between the space of  quasi-orbits in $\widehat{\mathfrak{u}}$  under $\Ad^*$ 
 and $\Char(U).$   
The description of $\Char(U)$ appears in  \cite[Theorem 4.2]{CareyEtAl1} 
and is also implicit in \cite[Proposition 2.7]{Pfeffer}.

\end{remark}

We now  rephrase Theorem~\ref{Theo-GenAlgGroup} in terms of  factor representations of $G$.
Recall that a \textbf{factor representation}  of a group $G$ is a unitary representation $\pi$ of $G$ on a Hilbert space $\H$ such the von Neumann subalgebra $\pi(G)''$ of $\L(\H)$ is a factor  (see  also Subsection~\ref{TracesAndVNalgebras}). Two  such representations $\pi_1$ and $\pi_2$ are said to be \textbf{quasi-equivalent} if  their exists an isomorphism $\Phi: \pi_1(G)''\to \pi_2(G)''$ such that 
$\Phi(\pi_1(g))= \pi_2(g)$ for every $g\in G.$
A factor representation
 $\pi$ of $G$  is said to be of \textbf{finite type} if 
$\pi(G)''$ is a finite factor,
that is, if $\pi(G)''$ admits a trace $\tau$; in this case,  $\tau\circ \pi$ belongs to $\Char(G)$ and the map  $\pi\mapsto \tau\circ \pi$ factorizes to a bijection between the quasi-equivalence classes of  factor representations  of finite type of $G$ and $\Char(G)$;
for all this, see \cite[Chap.6, Chap. 17]{DixmierC*}.

The following result follows immediately  from Theorem~\ref{Theo-GenAlgGroup},
in combination with Proposition~\ref{Pro-InducedTrace} and \cite[Corollary 6.8.10]{DixmierC*}.

Let $\Ga=L\ltimes N$  be a semi-direct product of a subgroup 
$L$ and an abelian normal subgroup $N$. Let $\sigma$ be a unitary representation of $L$ on a Hilbert space $\H$ 
and let $\chi\in \widehat{N}$ be such that
$\chi^g= \chi$ for every $g\in L.$ It is straightforward to check 
that $\chi \sigma$ defined by $\chi\sigma (g, n)= \chi(n) \sigma(g)$
for $(g,n)\in \Ga$ is a unitary representation of $\Ga$ on $\H$
\begin{theorem}
\label{Theo-GenAlgGroup-bis}
Let $G=LU$ be as in Theorem~\ref{Theo-GenAlgGroup}.
\begin{itemize}
\item[(i)] For every $ \la\in \widehat{\mathfrak{u}}$ and every  factor representation $\sigma$   of finite type of $L_\la,$ the representation $\pi_{(\la, \sigma)}:=\Ind_{L_\la P_\la}^G \chi_\la \sigma$ induced by $\chi_\la \sigma$ is a  factor representation  of finite type of $G$; moreover, every factor representation  of finite type of $G$ is quasi-equivalent
to a representation of the form $\pi_{(\la, \sigma)}$ as above.
\item[(ii)] Let $\la_1, \la_2\in \widehat{\mathfrak{u}}$ and $ \sigma_1,
 \sigma_2$  factor representations  of finite type of $L_{\la_1}, L_{\la_2},$ respectively.
 Then $ \pi_{(\la_1, \sigma_1)}$ and $ \pi_{(\la_2, \sigma_2)}$ are quasi-equivalent
 if and only if $\la_1$ and $\la_2$ have the same quasi-orbit under the coadjoint action $\Ad^*$ and  $\sigma_1$ and $\sigma_2$ are quasi-equivalent.
 \end{itemize}
 \end{theorem}

This paper is organized as follows. 
 In Section~\ref{S:Prelimilaries}, we  establish 
with some detail general facts about  functions of positive type on a group $\Ga$ which are invariant under a group of automorphisms of $\Ga.$ 
Section~\ref{S:CharUnip} deals with the crucial relationship 
(Proposition~\ref{Pro-TracesNilpotent}) between
 traces on an unipotent algebraic groups and invariant traces 
 on the associated  Lie algebra. 
In Section~\ref{S:CharInvariant}, we show how 
the study of characters on an algebraic group over $\QQ$ 
leads to the study of invariant probability measures on adelic solenoids.
Such measures as well as orbits closures are  classified
in Section~\ref{S:InvProba}. The proof of Theorem~\ref{Theo-GenAlgGroup}
is completed in Section~\ref{S:GenAlgGroup}.
In Section~\ref{S:Examples}, we compute  $\Char(G)$ for a few specific
examples of algebraic groups $G.$

\section{Invariant traces and von Neumann algebras}
\label{S:Prelimilaries}

We consider  functions of positive type on a group $\Ga$ which are invariant under
a group $G$ of automorphisms of $\Ga$, which may be larger than the group
of inner automorphisms of $\Ga.$
A systematic treatment of such functions is missing in the literature,
although  they  have already been considered in \cite{Thoma1} and \cite{Thoma2}.
As they play an important r\^ole throughout this article,  we establish with some detail some general  facts about them; in particular,  we give new (and, as we hope, more transparent) proofs for two crucial and non obvious properties of these functions (Theorems~\ref{Prop-GNS-InnerAut} and \ref{Prop-CharProductGroup}), based on the consideration of  the associated von Neumann algebras.
\subsection{Some general facts on invariant traces}
\label{SS:GeneralitiesTraces}

Let $\Gamma, G$ be  discrete groups and assume that $G$ acts by automorphisms on $\Ga.$ 
\begin{definition}  
\label{Def-Trace}
(i) A function $\varphi \colon \Ga \to \CC$ is called a
\textbf{$G$-invariant trace} on $\Ga$ if 
\begin{itemize}
\item $\varphi$ is of positive type, that is, for all $\la_1,\dots,\la_n \in \CC$ and all $\g_1,\dots,\g_n \in \Ga$, we have
$$\sum_{i,j=1}^n \lambda_i \bar \lambda_j \varphi(\g_j^{-1} \g_i) \geq 0,$$
\item $\vfi(g(\g))= \varphi(\g)$ for all $\gamma\in \Gamma$ and $g\in G$, and
\item $\vfi$ is normalized, that is, $\vfi(e)=1.$
\end{itemize}
We denote by $\Tr(\Ga, G)$  the set of $G$-invariant traces on $\Ga.$
In the case where $G=\Ga$ and $\Ga$ acts on itself by conjugation,
we write $\Tr(\Ga)$ instead of $\Tr(\Ga,\Ga).$

\noindent
(ii) The set $\Tr(\Ga, G)$ is  a compact convex set in the unit ball of $\ell^{\infty}(\Ga)$ endowed with the weak* topology. Let $\Char(\Ga, G)$ be the set of extremal
points in $\Tr(\Ga,G).$

In case $G=\Ga,$ we write as above, $\Char(\Ga)$ instead of $\Char(\Ga,\Ga)$.

\noindent 
(iii) Functions $\vfi \in\Char(\Ga, G)$ will be called
\textbf{$G$-invariant characters} on $\Ga$ and are characterized by the following property:
if $\psi$ is a $G$-invariant function of positive type on $\Ga$ which is dominated by $\vfi$
(that is, $\vfi-\psi$ is a  function of positive type), then $\psi=\la \vfi$ for some $\la>0$.

\end{definition}

\begin{remark}
\label{Rem-Choquet}
Assume that $\Ga$ is countable. Then $\Tr(\Ga, G)$ is metri\-za\-ble 
and  $\Char(\Ga, G)$ is a Borel subset of  $\Tr(\Ga, G).$
By Choquet's theory, every $\vfi\in \Tr(\Ga, G)$ can be written
as integral 
$$\vfi=\int_{\Tr(\Ga, G)} \psi d\mu_\vfi(\psi)$$ for a  probability 
measure $\mu_\vfi$ on  $\Tr(\Ga, G)$
with $\mu_\vfi(\Char(\Ga, G))=1.$ 
When $G$ contains the group
 of inner automorphisms of $\Ga,$ the measure 
 $\mu_\vfi$ is unique, as $\Tr(\Ga, G)$  is a Choquet simplex in this case (\cite{Thoma1}).
\end{remark}

The proof of the following proposition is straightforward.
Observe that, if $N$ is a $G$-invariant normal subgroup of $\Ga,$
then $G$ acts by automorphisms on the quotient group $\Ga/N.$
\begin{proposition}
\label{Pro-NormalSubgroup}
Let $N$ be a $G$-invariant normal subgroup of $\Ga$  
and let $p:\Ga\to \Ga/N$ be the canonical projection.
\begin{itemize}
\item[(i)] 
For every $\vfi \in \Tr(\Ga/N, G),$ we have $\vfi\circ p\in \Tr(\Ga,G).$
\item[(ii)] The image of the map 
$$ \Tr(\Ga/N, \Ga)\to \Tr(\Ga, G), \qquad \vfi \mapsto \vfi\circ p$$
is $\{\psi\in \Tr(\Ga, G)\mid \psi|_N=1_N\}.$
\item[(iv)] We have 
$\vfi \in \Char(\Ga/N, G)$ if and only  $\vfi\circ p\in \Char(\Ga, G).$
\end{itemize}
\end{proposition}

Let  $\vfi$ be a normalized function of positive type on $\Ga$. Recall  (see \cite[Theorem C.4.10]{BHV}) that there is a so-called \textbf{GNS-triple} $(\pi, \H, \xi)$ associated to $\vfi$, consisting  of a cyclic unitary representation  of $\Ga$  on a Hilbert space $\H$ with cyclic unit vector $\xi$ such that
$$
\vfi (\g) = \langle\pi (\g)\xi,  \xi\rangle \tout  \g\in \Ga.
$$
The triple $(\pi, \H, \xi)$ is unique in the following sense: if $(\pi', \H', \xi')$ is another GNS
triple associated to $\vfi,$ then there is a \emph{unique}\,  isomorphism $U:\H\to \H'$ of Hilbert spaces such that 
$$U \pi(\g) U^{-1}=\pi'(\g) \tout \g\in \Ga \et U\xi=\xi'.$$

As the next proposition shows, invariant traces on a subgroup of $\Ga$ can be induced to invariant  traces on $\Ga$.
 
For a function $\psi: Y\to \CC$ defined on a subset $Y$ of a set $X,$ 
we denote by $\widetilde{\psi}$ the \textbf{trivial extension}
of $\psi$ to $X,$ that is, the function  $\widetilde{\psi}: X\to \CC$  
given by 
$$
\widetilde{\psi}(x)=
\begin{cases}
\psi(x)&\text{if } x\in Y\\
0&\text{if } x\notin Y
\end{cases},
$$
\begin{proposition}
\label{Pro-InducedTrace}
Let $H$ be a $G$-invariant subgroup of $\Ga$
and $\psi\in \Tr(H, G).$  Then  $\widetilde{\psi}\in \Tr(\Ga,G).$
Moreover, if  $\sigma$ is a GNS representation of $H$ associated to $\psi$,
then  the GNS representation of $\Ga$ associated to $\widetilde{\psi}$
is equivalent to the induced representation $\Ind_H^\Ga \sigma$.
\end{proposition}
\begin{proof}
Set $\vfi:=\widetilde{\psi}.$ It is obvious that $\vfi$ is $G$-invariant.
The fact that $\vfi$ is a function of positive type can be checked directly from
the definition of such a function (see \cite[32.43]{HewittRoss2}).
As we need to identify the GNS representation associated to $\vfi,$ we sketch
another well-known proof for this fact.

Let $(\sigma, \K, \eta)$ be a GNS triple associated to $\psi.$
Let $\pi=\Ind_H^\Ga \sigma$ be realized on $\H=\ell^2(\Ga/H,\K),$ as in 
\cite[Remark 2, \S 6.1]{Folland}. Let $\xi\in \H$ be defined
by $\xi(H)= \eta$ and $\xi(\g H)=0$ if $\g\notin H.$ 
Then $\vfi(\g)= \langle \pi(\g) \xi, \xi\rangle$ for every $\g\in \Ga$ and 
$\xi$ is a cylic vector for $\pi.$ So, 
$(\pi, \H, \xi)$ is a GNS triple for $\vfi.$
\end{proof}

Attached to a given  invariant trace on $\Ga,$ there
are two invariant subgroups of $\Ga$ which will play an important role in the sequel.
 \begin{proposition}
\label{Proposition-KernelProjectiveKernel}
  Let  $\vfi\in \Tr(\Ga, G).$ Define
  $$K_\vfi=\{\g\in \Ga\mid \vfi(\g)=1\} \et P_\vfi=\{\g\in \Ga\mid |\vfi(\g)|=1\}.$$
  \begin{itemize}
  \item[(i)]  $K_\vfi$ and $P_\vfi$ are $G$-invariant closed subgroups of $\Ga$
  with $K_\vfi \subset P_\vfi$.
  \item[(ii)] For $x\in P_\vfi$ and $\g\in \Ga,$ we have $\vfi(x\g)= \vfi(x) \vfi(\g)$;
  in particular, the restriction of $\vfi$ to $P_\vfi$ is a $G$-invariant unitary character of $P_\vfi.$
  \item[(iii)] For $x\in P_\vfi$ and $g\in G,$ we have $g(x)x^{-1} \in K_\vfi.$
    \end{itemize}
 \end{proposition}
   \begin{proof}
   Let $(\pi, \H, \xi)$ be  GNS-triple associated to $\vfi.$ Using the equality case of Cauchy-Schwarz inequality,  it is clear that 
  $$
  K_\vfi=\{x\in \Ga\mid \pi(x)\xi=\xi\} \et P_\vfi=\{x\in \Ga\mid \pi(x)\xi=\vfi(x)\xi\}.
  $$
  Claims (i), (ii) and (iii) follow from this.
  
 \end{proof}

We will later need the following elementary lemma.
\begin{lemma}
\label{Lem-Hilbert}
  Let $\vfi\in \Tr(\Ga, G)$ and $\ga\in \Ga.$ Assume that there exists a sequence
 $(g_n)_{n\geq 1}$ in $G$ such that 
 $$\vfi(g_n(\ga)g_m(\ga)^{-1})=0\tout n\neq m.$$
 Then $\vfi(\ga)=0.$
  \end{lemma}
\begin{proof}
Let $(\pi, \H, \xi)$ be a GNS triple for $\vfi.$
We have   
$$
\begin{aligned}
\langle \pi( g_m(\ga)^{-1})\xi,\pi( g_n(\ga)^{-1})\xi\rangle
&=\langle \pi(g_n(\ga)g_m(\ga)^{-1})\xi,\xi\rangle\\
&=\vfi(g_n(\ga)g_m(\ga)^{-1}) =0.
\end{aligned}
$$
for all $m,n$ with $m\neq n$.
Therefore, $(\pi(g_n(\ga)^{-1})\xi)_{n\geq 1}$ 
is an orthonormal sequence in $\H$ and
so  converges weakly to $0.$ 
The claim follows,  since $\vfi(\ga)= \overline{\vfi(\ga^{-1})}$ and, for all $n,$ 
$$
\vfi(\ga^{-1})=\vfi(g_n(\ga)^{-1})=\langle \pi(g_n(\ga)^{-1})\xi,\xi\rangle.
$$
\end{proof}

\subsection{Invariant traces and von Neumann algebras}
\label{TracesAndVNalgebras}
We relate traces on groups to traces on appropriate von Neumann algebras.

Let $\Gamma, G$ be  discrete groups and assume that $G$ acts by automorphisms on $\Ga.$ 
This uniqueness property in the GNS construction for functions
of positive type  has the following consequence for $G$-invariant traces on $G$
\begin{proposition}
\label{Pro-GNS-Trace}
Let  $\vfi\in\Tr(\Ga, G)$ and let $(\pi, \H, \xi)$ be  a GNS-triple associated to $\vfi.$
There exists a unique unitary representation $g\mapsto U_g$ of $G$ on $\H$ such that $$U_g \pi(\g) U_g^{-1}=\pi(g(\g)) \tout g\in G, \ga\in \Ga \et U_g\xi=\xi.$$
\end{proposition}
\begin{proof}
Let $g\in G.$ Consider  the unitary representation $\pi^g$  of $\Ga$  on $\H$ given by 
$\pi^g(\g)=\pi(g(\g))$ for $\g\in \Ga.$
Since $\vfi$ is invariant under $g,$ the triple $(\pi^g, \H, \xi)$ is another GNS-triple associated to $\vfi$. Hence,  there exists a unique  unitary operator $U_g:\H\to \H$ such that 
$$U_g \pi(\g) U_g^{-1}=\pi^g(\g) \tout \g\in \Ga \et U_g\xi=\xi.$$
Using the uniqueness of $U_g$, one checks that  $g\mapsto U_g$ is 
a representation of $G.$
\end{proof}

 We now give a necessary and sufficient condition for a $G$-invariant trace on $\Ga$ to be
 a character.

Let $(\pi, \H, \xi)$ be  GNS-triple associated to $\vfi$
and  $g\mapsto U_g$ the unitary representation of $G$ on $\H$ 
as in \ref{Pro-GNS-Trace}.
Let  $\M_\vfi$ the von Neumann subalgebra of $\L(\H)$ generated 
 by  the set of operators $\pi(\Ga) \cup \{U_g\mid g\in G\}$, that is,
$$\M_\vfi:= \left\{\pi(\ga), U_g\mid \g\in \Ga, g\in G\right\}''.$$

\begin{proposition}
\label{Pro-GNS-VN}
Let  $\vfi\in\Tr(\Ga, G)$ with associated GNS-triple $(\pi, \H, \xi)$ and
$\M_\vfi$  the von Neumann subalgebra of $\L(\H)$ as above.
For every $T\in\L(\H)$ with $0\leq T\leq I$,  let $\vfi_T$ be defined by 
$\vfi_T(\g)=\langle \pi(\ga) T\xi,T\xi\rangle$ for $\g\in \Ga.$
Then $T\mapsto \vfi_T$ is a bijection between 
$\{T\in\M_\vfi'\mid 0\leq T\leq I\}$ and the set of $G$-invariant  functions of positive type on $\Ga$ which are dominated by $\vfi.$
In particular, we have $\vfi \in \Char(\Ga, G)$ if and only if $\M_\vfi'=\CC I.$
\end{proposition}
\begin{proof}
The map $T\mapsto \vfi_T$ is known to be a bijection between the set 
$\{T\in \pi(\Ga)'\mid 0\leq T\leq I\}$ and the set of  functions of positive type on $\Ga$
which are dominated by $\vfi$ (apply \cite[Proposition 2.5.1]{DixmierC*}  to the $*$-algebra $\CC[\Ga]$, with the convolution product and the involution given by $f^*(\ga)=f(\ga^{-1})$ for $f\in \CC[\Ga]$).

Therefore, it suffices to check that, for $T\in \pi(\Ga)'$ with $0\leq T\leq I,$
the function  $\vfi_T$ is $G$-invariant if and only if $T\in \{U_g\mid g\in G\}'$.

Let $T\in \M_\vfi'$ with $0\leq T\leq I.$  For every $g\in G,$ we have
$$
\begin{aligned}
\vfi_T(g(\g))&=\langle \pi(g(\g)) T\xi,T\xi\rangle = \langle U_g \pi(\g) U_{g^{-1}}T\xi,T\xi\rangle\\
&=\langle \pi(\g) TU_{g^{-1}}\xi,T U_{{g}^{-1}}\xi\rangle
=\langle \pi(\g) T\xi,T\xi\rangle\\
&= \vfi_T(\g),
\end{aligned}
$$
for all $\g\in \Ga;$ so, $\vfi_T$ is $G$-invariant.

Conversely, let $T\in \pi(\Ga)'$ with $0\leq T\leq I$ be such that $\vfi_T$ is $G$-invariant. 
Let  $g\in G$. For every $\g\in \Ga,$ we have
 $$
 \begin{aligned}
 \vfi_{U_{g^{-1}}TU_g}(\ga)&=\langle \pi(\g) U_{g^{-1}}TU_g\xi,U_{g^{-1}}TU_g\xi\rangle\\
 &=\langle  \pi(\g) U_{g^{-1}}T\xi,U_{g^{-1}}T\xi\rangle\\
 &=\langle U_g \pi(\g) U_{g^{-1}}T\xi,T\xi\rangle\\
 &=\langle  \pi(g(\g)) T\xi,T\xi\rangle\\
 &=\vfi_T(g(\g))= \vfi_T(\g).
 \end{aligned}
 $$
Hence, $U_{g^{-1}}TU_g$ is a scalar multiple of $T$, by uniqueness of $T.$
Since $U_g$ is unitary and $T\geq 0,$ it follows that $U_{g^{-1}}TU_g= T$ for all $g\in G$
and so $T\in \M_\vfi'.$
\end{proof}

Let $Z(\Ga)$  be the center of $\Ga.$ We call the  subgroup 
$$Z(\Ga)^G:=\{z\in Z(\Ga)\mid g(z)=z \text{ for all } g\in G\}$$
the \emph{$G$-center}
of $\Ga.$ 
We draw a first consequence on the values
taken by a  $G$-invariant character on $Z(\Ga)^G.$

\begin{corollary}
\label{Cor-CharCenter}
Let $\vfi\in \Char(\Ga, G).$
The $G$-center $Z(\Ga)^G$ of $\Ga$ is contained in $P_\vfi.$
\end{corollary}
\begin{proof}
Let $(\pi, \H, \xi)$ be  GNS-triple associated to $\vfi.$
For every $z\in Z(\Ga)^G,$ the operator
$\pi(z)$ commutes with $\pi(\ga)$ and $U_g$
for every $\ga\in \Ga$ and every $g\in G.$
It follows from Proposition~\ref{Pro-GNS-VN} that $\pi(z)$ is a scalar multiple of $I_{\H}$
and hence that  $z\in P_\vfi.$
\end{proof}

We will be mostly interested in the case where $G$ contains the group
 of all inner automorphisms of $\Ga.$
Upon replacing $G$ by the semi-direct group $G\ltimes \Ga,$ we may assume
without loss of generality that $\Ga$ is a normal subgroup of $G.$

Let $G$ be a discrete group and $N$ a normal subgroup of $G.$ 
Then $\Tr(N, G)\subset \Tr(N)$ denotes the convex set of $G$-invariant traces on $N$
and $\Char(N,G)$ the set of extreme points in $\Tr(N, G).$
We first draw a consequence of  Propositions~\ref{Pro-GNS-Trace} and \ref{Pro-GNS-VN} in the case $N=G.$

Recall that a (finite) trace on a von Neumann algebra $\M\subset \L(\H)$  is   a   positive linear functional $\tau$ on $\M$   such that
$$\tau (TS)= \tau (ST) \tout S,T\in \M.$$
Such a trace $\tau$ is faithful if $\tau(T^*T)>0$ for every $T\neq 0$
and normal if $\tau$ is continuous on the unit ball of $\M$
for the weak operator topology.
A von Neumann algebra $\M$ which has a normal faithful   trace is said to be a \textbf{finite von Neumann algebra}.

Let $\tau$ be a   normal faithful   trace on $\M$ and let $T$ be in the center $\M\cap \M'$ of $\M$ with $0\leq T\leq I.$  Then $\tau_T:\M\to \CC$, defined by 
$$\tau_T(S)= \tau (ST) \tout S\in \M,$$
 is a normal trace on $\M$ which is dominated by $\tau$
(that is,  $\tau_T(S)\leq \tau(S)$ for every $S\in \M$ with $S\geq 0$).
The map $T\mapsto \tau_T$ is a bijection
between $\{T\in \M\cap \M'\mid 0\leq T\leq I\}$ and the set of normal traces
  on $\M$ which are  dominated by  $\tau$  (see Theorem 3 in Chap. I, \S 6 of  \cite{DixmierVN}).

Recall that a von Neumann subalgebra $\M$ of $\L(\H)$ is
a \textbf{factor} if its center $\M\cap \M'$ consists only of multiples
of the identity operator $I.$

\begin{corollary}
\label{Pro-GNS-Trace-bis}
Let  $\vfi\in\Tr(G)$ and let $(\pi, \H, \xi)$ be  GNS-triple associated to $\vfi.$
\begin{itemize}
\item[(i)]  The linear functional $\tau:T\mapsto \langle T\xi, \xi\rangle$ is a normal
faithful trace on $\pi(G)''.$
\item[(ii)] The commutant $\M_\vfi'$ of $\M_\vfi$ coincides with the center of  the von Neumann algebra  $\pi(G)''$ generated by $\pi(G)$. In particular, 
 $\vfi\in\Char(G)$ if and only if $\pi(G)''$ is a factor.
\end{itemize}
\end{corollary}
\begin{proof}
(i) One checks immediately that $\tau$, as defined above, is a trace on $\pi(G)''$. It is clear that $\tau$ is normal. Let $T\in \pi(G)''$ be such that $\tau(T^*T)=0.$
Then 
$$\Vert T\pi(g)\xi\Vert^2=\tau\left(\pi(g^{-1})T^*T\pi(g)\right)=\tau(T^*T)=0,$$
 that is, $T\pi(g)\xi=0$ for all
$g\in G;$ hence, $T=0$ since $\xi$ is a cyclic vector for $\pi.$ So, $\tau$ is faithful.

\noindent
(ii) 
 Observe first that, for every $g\in G,$ we have
 $$U_g \pi(x)U_{g^{-1}}= \pi(gx g^{-1})= \pi(g) \pi(x) \pi(g^{-1}) \tout  x\in G,$$
where $g\mapsto U_g$ is the representation of $G$ as in  Proposition~\ref{Pro-GNS-Trace}. It follows that $U_g TU_{g^{-1}}= \pi(g)T\pi(g^{-1})$ for every $T\in \pi(G)''.$
Hence,  $\pi(G)''\cap \pi(G)'$ is contained in $\M_\vfi'.$

Conversely, let  $T\in \M_\vfi'$ with $0\leq T\leq I.$
By Proposition~\ref{Pro-GNS-VN},  $\vfi_{T^{1/2}}$ 
is a $G$-invariant function of positive type dominated by $\vfi.$
The canonical extension of $\vfi_{T^{1/2}}$ to $\pi(G)''$  is a normal
trace $\tau'$ on $\pi(G)''.$ Hence, by the result recalled above, 
$\tau'= \tau_S$ for  a unique $S\in \pi(G)'\cap \pi(G)''$ with $0\leq S\leq I.$
This shows that $\vfi_{T^{1/2}}=\vfi_{S^{1/2}}$. Since $T$ and $S$ belong
both to $\pi(G)',$ it follows that $T=S.$ So, $T\in \pi(G)'\cap \pi(G)''$.
Therefore, $\M_\vfi'$ is contained in $\pi(G)'\cap \pi(G)''.$
  \end{proof}
  
  The following  result, which  will be crucial in the sequel,
appears  in \cite[Lemma 14]{Thoma1};  the proof we give here for it is shorter and more transparent than the original one.
\begin{theorem}
\label{Prop-GNS-InnerAut}
Let $G$ be a discrete group, $N$ a normal subgroup of $G$
and $\psi\in \Char(G).$ Then $\psi|_N\in \Char(N, G).$
\end{theorem}
\begin{proof}
Let $(\pi, \H, \xi)$ be  GNS-triple associated to $\psi.$
Set $\vfi:=\psi|_N$ and let $\K$ be the closed linear span of $\{\pi(x)\xi\mid x\in N\}.$
Then $(\pi|_N, \K, \xi)$ is a   GNS-triple associated to $\vfi.$

Let $g\mapsto U_g$ be the representation of $G$ on $\H$ associated
to $\psi$ as in  Proposition~\ref{Pro-GNS-Trace}.
The subspace $\K$ is invariant under 
$U_g$ for $g\in G$, since $U_g \pi(x) U_{g}^{-1}=\pi(g xg^{-1})$ and $U_g\xi=\xi.$
So, the representation of $G$ on $\K$ associated to $\vfi$
is $g\mapsto U_g|_{\K}.$ 
Let $\M_\vfi'$ be the von Neumann subalgebra
of $\L(\K)$ generated by  
$$\{\pi(x)|_{\K}\mid x \in N\} \cup \{U_g|_{\K}\mid g\in G\}.$$

In view of Proposition~\ref{Pro-GNS-VN},  it suffices to show  that
$\M_\vfi'=\CC I.$  
Let   $T\in \M_\vfi'$ with $0\leq T\leq I.$ 
Consider the linear functional
$\tau'$ on $\pi(G)''$ given by 
$$\tau' (S)= \langle S T\xi, T\xi\rangle \tout S\in \pi(G)''.$$
 We claim that $\tau'$ is a normal trace on 
$\pi(G)''.$ Indeed, it is clear that $\tau$ is normal; moreover, for $g,h\in G,$ we have
$$
\begin{aligned}
\tau'(ghg^{-1})&= \langle \pi(ghg^{-1}) T\xi, T\xi \rangle\\
&=\langle U_g\pi(h) U_{g^{-1}}T\xi, T\xi \rangle\\
&=\langle  \pi(h) TU_{g^{-1}}\xi, U_{g^{-1}}T\xi \rangle\\
&=\langle  \pi(h) TU_{g^{-1}}\xi, TU_{g^{-1}}\xi \rangle\\
&=\langle  \pi(h) T\xi, T\xi \rangle= \tau' (h).
\end{aligned}
$$
Let $\tau'': = \tau+ \tau',$ where 
$\tau$ is the faithful trace on $\pi(G)''$ defined by $\vfi,$ as in Corollary~\ref{Pro-GNS-Trace-bis}.
Then $\tau''$ is a normal faithful trace on $\pi(G)''$ and $\tau''$ dominates $\tau$ and $\tau'.$
Since $\pi(G)''$ is a factor, it follows that $\tau$ and $\tau'$ are both proportional
to $\tau''.$ Hence, there exists $\la\geq 0$ such that $\tau'= \la \tau.$ So,
 $$\langle  \pi(x) T\xi, T\xi \rangle= \langle  \pi(x) \sqrt{\la} \xi, \sqrt{\la} \xi \rangle \tout x\in N.$$
 Since $T\in \{\pi(x)|_{\K}\mid x\in N\}'$ and $0\leq T\leq I,$ it follows that $T=\sqrt{\la} I_{\K}.$
 
\end{proof}

As we now show,   the set of characters of a product group admits a simple description;
again, this is  a  result  due to Thoma (\cite[Satz 4]{Thoma2}) for which we provide
a short proof.

For sets $X_1, \dots, X_r$ and functions $\vfi_i: X_i\to \CC, i\in\{1,\dots, r\},$
we denote by $ \vfi_1\otimes \cdots \otimes\vfi_r$ the function 
on $X_1\times \dots \times X_r$ given by 
$$
 \vfi_1\otimes \cdots \otimes\vfi_r (x_1, \dots, x_r) =  \vfi_1(x_1)\cdots \vfi_r (x_r), 
 $$
 for all $(x_1,\dots, x_r)  \in X_1\times \dots \times X_r.$
\begin{theorem} 
\label{Prop-CharProductGroup}
Let $G_1, G_2$ be discrete groups.
Then 
$$\Char(G_1\times G_2)=\left\{ \vfi_1\otimes \vfi_2\mid \vfi_1\in \Char(G_1), \vfi_2\in \Char(G_2)\right\}.
$$
\end{theorem}
\begin{proof}
Set $G:=G_1\times G_2.$ 

For $i=1,2,$ let $\vfi_i\in \Char(G_i).$ We claim that  
$$\vfi:=\vfi_1 \otimes \vfi_2\in \Char(G).$$ 
Indeed, let $(\pi_i, \H_i, \xi_i)$ be a GNS triple associated to $\vfi_i.$
Then $(\pi, \H , \xi)$ is   a GNS triple associated to $\vfi,$ where $\pi$ is the tensor product representation $\pi_1\otimes \pi_2$
on $\H:=\H_1\otimes \H_2$ and $\xi:=\xi_1\otimes \xi_2.$
In view of Proposition~\ref{Pro-GNS-Trace-bis}, we have to show that 
$\pi(G)''$ is a factor. For this, it suffices to show that  the von Neumann algebra 
$\M$ generated by $\pi(G)''\cup \pi(G)'$ coincides with $\L(\H).$

On the one hand,  $\pi(G)''$ contains $\pi_1(G_1)''\otimes I$ and $I\otimes \pi_2(G_2)''$, and $\pi(G)'$ contains $\pi_1(G_1)'\otimes I$ and $I\otimes \pi(G_2)'$; hence,
$\M$ contains $\M_1\otimes \M_2,$ where $\M_i$ is the von Neumann algebra
generated by $\pi_i(G_i)'' \cup \pi_i(G_i)'.$ 
On the other hand, since $\vfi_i\in \Char(G_i),$
we have $\M_i=\L(\H_i).$  So, $\M$ contains the von Neumann algebra generated by 
$\{T_1\otimes T_2\mid T_1\in \L(\H_1), T_2\in \L(\H_2)\},$ which is $\L(\H).$

Conversely, let $\vfi \in \Char(G).$ Let  $(\pi, \H , \xi)$ be a GNS triple associated to $\vfi.$ By Proposition~\ref{Pro-GNS-Trace-bis}, $\M:=\pi(G)''$ is a factor.

For $i=1,2,$ set $\M_i:= \pi(G_i)''$, where we identify  $G_i$ with the 
subgroup $G_i\times\{e\}$ of $G.$
We claim that $\M_1$ and $\M_2$ are factors. Indeed, 
since $\M_1\subset \M_2'$,  the center $\M_1\cap \M_1'$ of $\M_1$ is contained
in $\M_2'\cap \M_1'$. As $\M_1\cup \M_2$ generate $\M,$ it follows
that $\M_1\cap \M_1'$ is contained in $\M'$ and so in $\M\cap \M'.$
Hence, $\M_1\cap \M_1'=\CC I,$ since $\M$ is a factor.
So, $\M_1$ and, similarly, $\M_2$  are factors.

Next, recall (Corollary~\ref{Pro-GNS-Trace-bis}) that $\M$ has a normal faithful trace $\tau$ given by $\tau(T)= \langle T\xi, \xi\rangle$ for $T\in \M.$
The restriction $\tau^{(1)}$ of $\tau$ to $\M_1$ is a normal faithful trace on $\M_1.$

Let $T_2\in \M_2$ with $0\leq T_2 \leq I$ and $T_2\neq 0.$
Define  a positive and normal linear functional $\tau^{(1)}_{T_2}$ on $\M_1$
by 
$$
\tau^{(1)}_{T_2}(S)= \tau (ST_2) \tout  S\in \M_1.
$$
For $S,T\in \M_1,$ we have
$$
\begin{aligned}
\tau^{(1)}_{T_2}(ST)= \tau (STT_2)= \tau (S T_2 T)
= \tau ((S T_2) T)= \tau(T (ST_2)) 
= \tau^{(1)}_{T_2}(TS).
\end{aligned}
$$
So, $\tau^{(1)}_{T_2}$ is a normal  trace on $\M_1.$
Clearly,  $\tau^{(1)}_{T_2}$ is dominated by $\tau^{(1)}.$ 
Since $\M_1$ is a factor, it follows from the result quoted 
before Corollary~\ref{Pro-GNS-Trace-bis} that 
there exists a scalar $\la(T_2)\geq 0$ such that $\tau^{(1)}_{T_2}=\la(T_2) \tau^{(1)}$, that is,
$$
\tau (T_1 T_2)= \la(T_2) \tau(T_1) \tout T_1\in \M_1.
$$
Taking $T_1=I,$ we see that $\la(T_2)= \tau(T_2).$
It follows that 
$$
\tau (T_1 T_2)= \tau(T_1) \tau(T_2) \tout T_1\in \M_1, T_2\in \M_2
$$
and in particular $\vfi= \vfi_1\otimes \vfi_2,$ for
$\vfi_i= \vfi|_{G_i}.$

\end{proof}

 The following result is an immediate consequence of Proposition~\ref{Prop-CharProductGroup}.

 Recall (see Proposition~\ref{Pro-NormalSubgroup}) that, when $N$ is a normal subgroup of a group $G,$ we can identify
 $\Char(G/N)$ with the subset $\{ \vfi \in \Char(G)\mid \vfi|_N=1\}$
 of $\Char(G)$.
 
\begin{corollary}
  \label{Cor-CharProduct}
 For discrete groups $G, G_1, \dots, G_r$, let 
 $$p: G_1\times \dots \times G_r\to G$$ be a surjective homomorphism. 
  Then 
  $$
  \Char(G)=\left\{ \vfi=\vfi_1\otimes \cdots \otimes\vfi_r\mid \vfi|_N=1 \, \text{and}\, 
  \vfi_i \in \Char(G_i), i=1,\dots, r\right\},
 $$
 where $N$ is the kernel of $p.$
  In particular, for  $\vfi\in \Char(G),$ we have
  $$
  \vfi(g_1 \dots g_n)= \vfi(g_1)\dots \vfi(g_n)
 $$
 for all $g_i\in p( \{e\}\times \cdots\times G_i\times \cdots, \times \{e\}).$
 \end{corollary}

 \section{Traces on unipotent groups}
\label{S:CharUnip}
In this section, we will show that traces on a unipotent algebraic group $U$
are in a one-to-one correspondence with $\Ad(U)$-invariant positive definite functions on
the Lie algebra of $U.$
\subsection{Invariant traces on abelian groups}
\label{SS:TracesAbelianGroups}
Let $A$ be a discrete abelian  group
 and $\widehat{A}$ the Pontrjagin dual of $A,$ which is a compact abelian group. 
 Then $\Tr(A)$ is the set of normalized functions of positive type on $A$
 and $\Char(A)= \widehat{A}$.
 
  Let $\Prob(\widehat{A})$ denote the 
set of regular probability measures on the Borel subsets of $\widehat{A}$.
For $\mu\in \Prob(\widehat A),$ the Fourier--Stieltjes transform  $\mathcal{F}(\mu):A\to \CC$
of $\mu$ is given by
$$
\mathcal{F}(\mu)(a)= \int_{\widehat{A}}\chi(a) d\mu(\chi)
\tout
\chi \in \widehat A.
$$
By Bochner's theorem (see e.g. \cite[\S 33]{HewittRoss2}),
 the map $\mathcal{F}: \mu\mapsto \mathcal{F}(\mu)$ is a bijection between 
$\Prob(\widehat{A})$ and $\Tr(A)$.

Let $G$ be a group acting by automorphisms on $A.$
Then $G$   acts by continuous automorphisms on 
$\Tr(A)$ and on $\widehat{A}=\Char(A)$, via the dual action
given by 
$$\vfi^g(a)=\vfi(g^{-1}(a)) \tout \chi\in \Tr(A), \, g\in G, \, a\in A.$$
Let $(g,\mu)\mapsto g_*(\mu)$ be the induced action of $G$ on $\Prob(\widehat{A});$
so, $g_*(\mu)$ is the image of $\mu\in \Prob(\widehat{A})$ under  
the map $\chi\mapsto \chi^g.$ 

Let $\Prob(\widehat{A})^G$  be the subset of $\Prob(\widehat{A})$
consisting of $G$-invariant probability measures and denote
by $\Prob(\widehat{A})^G_{\rm erg}$ the measures in $\Prob(\widehat{A})^G$ 
which are ergodic. 

\begin{proposition}
\label{Prop-Abelian}
Let $A$ be a discrete abelian group and $G$ a group acting by automorphisms
on $A.$ 
The Fourier-Stieltjes transform  $\mathcal{F}$ restricts
to bijections $\mathcal{F}: \Prob(\widehat{A})^G \to \Tr(A, G)$ and 
$\mathcal{F}: \Prob(\widehat{A})^G_{\rm erg} \to \Char(A, G).$ 
\end{proposition}
\begin{proof}
The claims follow from the fact that 
 $\mathcal{F}: \Prob(\widehat{A}) \to \Tr(A)$  is  an affine $G$-equivariant  map
 and that $\Prob(\widehat{A})^G_{\rm erg}$ is the set of extreme
points in the convex compact set $\Prob(\widehat{A})^G$. 
\end{proof}

\subsection{Invariant traces on unipotent groups}
\label{SS:TracesNipotentGroups}
Let $k$ be a field of characteristic $0$.  Let  $U_n$ be the group of upper triangular unipotent $n\times n$ matrices over $k,$ for $n\geq 1.$
Then $U_n$ is  the group of $k$-points of an algebraic group over $k$ and 
its  Lie algebra is the Lie algebra
 $\mathfrak{u}_n$ of the strictly upper triangular matrices.
 The  exponential map $\exp: \mathfrak{u}_n\to U_n$ is a bijection and,
 by the Campbell-Hausdorff formula, there exists a polynomial map 
 $P: \mathfrak{u}_n\times \mathfrak{u}_n\to \mathfrak{u}_n$ with coefficients
 in $k$ such that 
$\exp(X)\exp(Y)= \exp(P(X, Y))$ for all $X,Y\in \mathfrak{u}_n$.
Denote by $\log: U_n\to \mathfrak{u}_n$ the inverse map of $\exp.$

 Let  $\mathfrak{u}$ be a  nilpotent Lie algebra over $k.$
Then, by the theorems of Ado and Engel, $\mathfrak{u}$ can be viewed as
Lie subalgebra of $\mathfrak{u}_n$ for some $n\geq 1$ and  $\exp(\mathfrak{u})$ is an algebraic subgroup of $U_n$.

Let $U$ be the group of $k$-points of a  unipotent algebraic  group over $k,$
that is, an algebraic subgroup of $U_n$ for some $n\geq 1.$
Then $\mathfrak{u}=\log(U)$ is a Lie subalgebra of $\mathfrak{u}_n$
and $\exp: \mathfrak{u}\to U$ is a bijection  (for all this, see \cite[Chap.14]{Milne}).

 For every $u\in U,$ the automorphism of $U$ given by conjugation with 
$u$ induces an automorphism $\Ad (u)$ of the Lie algebra $\mathfrak{u}$
determined by  the property 
$$
\exp(\Ad(u)(X))= u\exp(X) u^{-1} \tout X\in \mathfrak{u}.
$$
Observe that a function $\vfi$ on $U$ is central (that is, constant on  the $U$ conjugacy classes) if and only if the corresponding function $\vfi\circ \exp$ on $\mathfrak{u}$
is $\Ad(U)$-invariant. 

The following  proposition  will be a crucial tool in our  proof of Theorem~\ref{Theo-GenAlgGroup} 
\begin{proposition}
\label{Pro-TracesNilpotent}
Let $U$ be the group of $k$-points of a  unipotent algebraic  group over a field $k$ of characteristic zero. Let $\vfi:U\to \CC.$ Then 
$\vfi\in \Tr(U)$ if and only if $\vfi\circ \exp\in \Tr(\mathfrak{u}, \Ad(U)).$
So, the map   $\vfi\mapsto \vfi\circ\exp$ is a bijection between 
$\Tr(U)$ and $\Tr(\mathfrak{u}, \Ad(U)).$
\end{proposition}

\begin{proof}

Set $\vfi':=\vfi\circ \exp.$ Since $\vfi$ and $\vfi'$ are invariant,
we  have to show that $\vfi$ is of positive type on $U$ if and only 
$\vfi'$ is of positive type on $\mathfrak{u}.$

Let $Z(U)$ be the center of $U$ and $\mathfrak{z}$ the center of
$\mathfrak{u}.$ Set $\chi:= \vfi|_{Z(U)}$ and 
$\chi':=\vfi'|_{\mathfrak{z}}.$

\vskip.2cm
$\bullet$ {\it First step.} Assume  that $\vfi$ is of positive type on $U$  \emph{or}
that $\vfi'$ is of positive type on $\mathfrak{u}.$
Then $\chi$ is of positive type on $Z(U)$ \emph{and}
 $\chi'$ is of positive type on $\mathfrak{z}.$

Indeed, this follows from the fact that  $\exp: \mathfrak{z} \to Z$ is a group isomorphism.

\vskip.2cm
We will  reduce the proof of Proposition~\ref{Pro-TracesNilpotent}
to the case where $\vfi$  has the following multiplicativity property
$$
\leqno{(*)}\quad \vfi(g z)= \vfi(g)\chi(z) \tout g\in U, z\in Z(U).
$$
Observe that property $(*)$ is equivalent to 
$$\leqno{(*)'}\quad \vfi'(X+Z)= \vfi'(X)\chi'(Z) \tout X\in \mathfrak{u}, Z\in \mathfrak{z},$$
since $\exp(X+Z)=\exp(X)\exp(Z)$ for  $X\in \mathfrak{u}$ and $Z\in \mathfrak{z}.$

\vskip.2cm
$\bullet$ {\it  Second  step.} To prove  Proposition~\ref{Pro-TracesNilpotent},
we may assume that $\vfi$ has property $(*)$.

Indeed, since $\Tr(U)$ is the closed convex hull of $\Char(U)$ and 
$\Tr(\mathfrak{u}, U)$ is the closed convex hull of $\Char(\mathfrak{u}, U),$
 it suffices to prove that if $\vfi\in \Char(U)$ then $\vfi'$ is of positive type  on $\mathfrak{u}$ and that if $\vfi'\in \Char(\mathfrak{u}, U)$ then $\vfi$ is of positive type on $U.$
 Moreover, by Corollary~\ref{Cor-CharCenter}
 and Proposition~\ref{Proposition-KernelProjectiveKernel},
$\vfi$ has property $(*)$ if $\vfi\in \Char(U)$ and if 
$\vfi'$ has property $(*)'$ if $\vfi'\in \Char(\mathfrak{u}, U)$. This proves the claim.

 \vskip.2cm
 In view of  the second step, we may and will assume
 in the sequel that  $\vfi:U\to \CC$ is a central function, normalized by $\vfi(e)=1,$ 
 with property $(*)$.
 If, moreover,  either $\vfi$ is of positive type or $\vfi'$ is of positive type, then 
 $\chi\in \widehat{Z(U)}$ and $\chi'\in\widehat{\mathfrak{u}}$, by the first step.
 
 \vskip.2cm
$\bullet$ {\it  Third step.} 
Assume that either $\vfi$ is  of positive type  or that $\vfi'$ is of positive type.
If $\ker \chi'$ contains no non-zero linear subspace, then both
$\vfi$ and  $\vfi'$  are of positive type.

We claim  that $\vfi' =\widetilde{\chi'}$ (that is, $\vfi'=0$ on $\mathfrak{u}\setminus \mathfrak{z}$). Since this is equivalent to $\vfi=\widetilde{\chi}$ 
and since $\chi\in \widehat{Z(U)}$ and  $\chi'\in \widehat{\mathfrak{z}}$,
once proved, this claim combined with Proposition~\ref{Pro-InducedTrace}
will imply that $\vfi$ and $\vfi'$ are of positive type.

Let $(\mathfrak{z}^i)_{1\leq i\leq r}$ be the ascending central series of 
$\mathfrak{u}$; so, $\mathfrak{z}^1= \mathfrak{z}$, 
$\mathfrak{z}^{i+1}$ is the inverse image in $\mathfrak{u}$ of the
center of $\mathfrak{u}/\mathfrak{z}^i$ under the canonical
map $\mathfrak{u}\to \mathfrak{u}/\mathfrak{z}^i$ for every $i$, and
$\mathfrak{z}^r= \mathfrak{u}.$

Let $(Z^i(U))_{1\leq i\leq r}$ be the corresponding  ascending central series of 
$U$ given by $Z^i(U)= \exp \mathfrak{z}^{i}$.

We show by induction on $i$ that 
 $\vfi' =0$ on $\mathfrak{z}^i\setminus \mathfrak{z}$
 for every $i\in \{2, \dots, r\}.$

Indeed, let $X\in \mathfrak{z}^2\setminus \mathfrak{z}.$
There exists $Y\in \mathfrak{u}$ with $[Y, X]\neq 0$. Since 
$[Y, X]\in \mathfrak{z}$ and since $\ker \chi'$ contains no non-zero linear subspace, 
there exists $t\in k$ such that   $\chi'(t[Y,X])=\chi'([tY,X])\neq 1$.
Upon replacing $Y$ by $tY,$ we can assume that 
 $\chi'([Y,X])\neq 1$. Since $\vfi'$ is $\Ad(U)$-invariant, it follows from property (*)
 that 
 $$
 \begin{aligned}
 \vfi'(X) &= \vfi'(\Ad(\exp Y)(X))= \vfi'(X+[Y,X])=\vfi'(X) \chi'([Y,X]).
 \end{aligned}
 $$
 As $\chi'([Y,X)] \neq 1,$ we have $\vfi'(X)=0$;  so,  the case $i=2$ is  settled.
 
 Assume now
 $\vfi'=0$ on $\mathfrak{z}^i\setminus \mathfrak{z}$
 for  some  $i\in \{2, \dots, r\}.$ 
 Let  $X\in \mathfrak{z}^{i+1}\setminus \mathfrak{z}^i.$
 Then there exists $Y\in \mathfrak{u}$ such that 
 $[Y,X]\notin\mathfrak{z}^{i-1}$.
  Let $(t_n)_{n\geq 1}$ be a sequence of pairwise distinct 
 elements in $k$. Set $y_n= \exp(t_n Y)\in U.$
 Denoting by $p_{i-1}: \mathfrak{u}\to \mathfrak{u}/ \mathfrak{z}^{i-1}$
 the canonical projection, we have
 $$
  \begin{aligned}
p_{i-1}(\Ad(y_n)X-X)&=p_{i-1}(-[t_n Y, X]),
  \end{aligned}
  $$
  since $[\mathfrak{u}, [\mathfrak{u},X]]\subset \mathfrak{z}^{i-1}.$ 
  In particular, we have
 $$(\Ad(y_n)X-X)- (\Ad(y_m)X-X)\notin \mathfrak{z} \tout n\neq m.$$
 Since $\Ad(y_n)X-X \in \mathfrak{z}^{i}$ and $\vfi'=0$ on $\mathfrak{z}^{i}\setminus \mathfrak{z}$ by the induction hypothesis, we have therefore
   $$
  \leqno{(**)'} \quad\vfi'(\Ad(y_n)X- \Ad(y_m)X)=0 \tout n\neq m.
  $$
 We also have,  by the Campbell-Hausdorff formula,
  $$
  \begin{aligned}
p_{i-1}(\log([y_n, \exp(X)]))&=p_{i-1}\left (\log(\exp(-t_n Y) \exp (X)  \exp(t_n Y) \exp(-X))\right)\\
&=p_{i-1}(-[t_n Y, X]),
  \end{aligned}
  $$
  where $[u,v]= uvu^{-1} v^{-1}$ is the commutator of $u,v\in U.$
   As $[t_n Y, X]$ commutes with $[t_m Y, X],$ it follows that 
 $$ [y_n, \exp(X)][y_m, \exp(X)]^{-1} \notin Z(U) \tout n\neq m.$$
 Since $[y_n, \exp(X)]\in Z^i(U)$ and $\vfi=0$ on $Z^{i}(U)\setminus Z(U)$ by the induction hypothesis, we have therefore
$$
  \leqno{(**)} \quad \vfi ([y_n, \exp(X)][y_m, \exp(X)]^{-1})=0 \tout n\neq m.
  $$
   If $\vfi'$ is of positive type,  it follows from Lemma~\ref{Lem-Hilbert} and from 
  $(**)'$  that $\vfi'(X)=0.$
  If $\vfi$ is of positive type, then Lemma~\ref{Lem-Hilbert} and 
  $(**)$ imply  that $\vfi(\exp X)=0,$ that  is, $\vfi'(X)=0.$
  
 As a result,  $\vfi' =0$ on $\mathfrak{z}^i\setminus \mathfrak{z}$
 for every $i\in \{2, \dots, r\}.$ Since $\mathfrak{z}^r= \mathfrak{u},$
 the claim is proved.

 \vskip.2cm
$\bullet$ {\it Fourth step.}  
Assume that either $\vfi$ is  of positive type  or that $\vfi'$ is of positive type.
Then both $\vfi$ and  $\vfi'$  are of positive type.

We proceed by induction on $\dim_k\mathfrak{u}$.
The case $\dim_k\mathfrak{u}=0$ being obvious, assume that the claim
is true for every unipotent algebraic group with a Lie algebra
of dimension strictly smaller than $\dim_k\mathfrak{u}$.

 In view of the third step, we may assume that there exists a subspace $\mathfrak{k}$ of $\mathfrak{z}$ with $\dim_k \mathfrak{k}>0$ contained in  $\ker \chi'.$
Then $\vfi'$ can be viewed as a function on 
the nilpotent Lie algebra $\mathfrak{u}/\mathfrak{k}$  and 
$\vfi$ as a function of positive type on the corresponding unipotent  algebraic
group $U/\exp(\mathfrak{k})$.  Since $\dim_k \mathfrak{u}/\mathfrak{k}$ is strictly
smaller than $\dim_k\mathfrak{u},$ the claim follows from the induction hypothesis.

\end{proof}

\begin{remark}
\label{Rem-Pro-TracesNilpotent}
Using induction on $\dim_k \mathfrak{u}$ as well as 
 the fact established in the third step of the proof of 
Proposition~\ref{Pro-TracesNilpotent}, one can easily obtain the description
of  $\Char(U)$ given in Theorem~\ref{Theo-GenAlgGroup} for the special case
$G=U$; in fact,  $\Char(U)$ (respectively, the primitive ideal space 
of $U$)  is determined in \cite{CareyEtAl1} (respectively, in \cite{Pfeffer})
using such arguments.
\end{remark}

Let $G$ be a group acting by automorphisms on $U.$
Every $g\in G$  induces an automorphism $ X\mapsto g(X)$ of  $\mathfrak{u}$
 determined by  the property
$$
\exp(g(X))= g(\exp(X)) \tout X\in \mathfrak{u}.
$$

Let $\widehat{\mathfrak{u}}$ be the Pontrjagin dual of the additive group $ \widehat{\mathfrak{u}}.$ Then $G$ acts by automorphisms $\widehat{\mathfrak{u}},$
induced by the dual action. 

Since the map $\psi \mapsto\psi\circ \log$ from the space of 
functions on  $\mathfrak{u}$ to the space of functions on $U$
is tautologically $G$-equivariant, the
following result is an immediate consequence of Propositions~\ref{Pro-TracesNilpotent} and \ref{Prop-Abelian}.

\begin{corollary}
\label{Cor-Pro-TracesNilpotent}
Let $U$ be as in Proposition~\ref{Pro-TracesNilpotent} and 
let $G$ be a group acting as automorphisms of  $U.$ Assume that 
the image of $G$ in $\Aut(\mathfrak{u})$ contains $\Ad(U)$.
\begin{itemize}
\item [(i)] The map 
$$ \Char(\mathfrak{u}, G)  \to \Char(U, G) \to\qquad  \psi \mapsto\psi\circ \log$$ 
is a bijection.
\item[(ii)] The map 
$$\Prob(\widehat{\mathfrak{u}})_{\rm erg}^G\to \Char(U,G), \qquad \mu\mapsto \mathcal{F}(\mu)\circ\log$$
is a bijection.
\end{itemize}
\end{corollary}

Let $G$ be a group of  automorphisms of  $U$ containing $\Ad(U).$
Let $\lambda\in \widehat{\mathfrak{u}}$. Recall (see Section~\ref{S0}) that we associated 
to $\lambda$ the following  two $G$-invariant
ideals   of $\mathfrak{u}$ 
$$\mathfrak{k}_\la=\{X\in \mathfrak{u}\mid \lambda(\Ad(g)(tX))=1 \tout g\in G, t\in k\}$$
and 
$$\mathfrak{p}_\la=\{X\in \mathfrak{u}\mid \lambda(\Ad(g)(tX))=\Ad(g)(tX) \tout g\in G, t\in k\}.$$
\begin{proposition}
\label{Pro-KCenter}
Let $\lambda\in \widehat{\mathfrak{u}}$, 
$p:\mathfrak{u}\to \mathfrak{u}/\mathfrak{k}_\lambda,$
and $P_\la= \exp \mathfrak{p}_\lambda.$
\begin{itemize}
\item[(i)] We have 
$$\mathfrak{p}_\lambda= p^{-1}\left(Z( \mathfrak{u}/\mathfrak{k}_\lambda)^G\right),$$ where $Z( \mathfrak{u}/\mathfrak{k}_\lambda)^G$ is the  central ideal  of $G$-fixed elements in $ \mathfrak{u}/\mathfrak{k}_\lambda.$
\item[(ii)]  The map 
$$\chi_\la: P_\la\to \SS^1, \qquad \exp(X)\to \lambda(X)$$
 is a $G$-invariant unitary character of $P_\la$.
 \end{itemize}
 \end{proposition}
 \begin{proof}
 (i) Let $X\in \mathfrak{u}.$ We have
 $$
\begin{aligned}
p(X) \in Z(\mathfrak{u}/\mathfrak{k}_\lambda)^G &
\Longleftrightarrow  \Ad(g)X-X\in \mathfrak{k}_\lambda \tout  g\in G\\
&\Longleftrightarrow  \Ad(g)(tX)-tX\in \mathfrak{k}_\lambda \tout  g\in G, t\in k \\
&\Longleftrightarrow  \lambda(\Ad(g) (tX))= \lambda(tX) \quad  \tout  g\in G, t\in k\\
&\Longleftrightarrow X\in \mathfrak{p}_\lambda. 
\end{aligned}
$$

(ii) is a  special case of Proposition~\ref{Pro-TracesNilpotent}.
\end{proof}
We will later need the following elementary lemma.
\begin{lemma}
\label{Lem-CommutarorAndSubgroup}
Let $U$ be as in Proposition~\ref{Pro-TracesNilpotent} and 
$g\in \Aut(U).$ Let $N$ be a normal subgroup of $U.$
For $X\in \mathfrak{u},$ the set 
$$
A:=\{t\in k\mid \exp(-tX)\exp(g(tX))\in N\}
 $$
 is a  subgroup of the additive group of the field $k.$ 
  \end{lemma}
 \begin{proof}
 Observe first that $0\in A$. Let 
 $t,s\in A.$ Then 
 $$
 \begin{aligned}
 &\exp(-(t-s)X)\exp(g((t-s)X))\\
 &=\exp(sX) \exp(-tX)\exp(g(tX)) \exp(g(-sX)\\
&=\exp(sX)( \exp(-tX)\exp(g(tX)) \exp(g(-sX)\exp(sX)) \exp(-sX)\\
&=\exp(sX)\left(\exp(-tX)\exp(g(tX))(\exp(-sX)\exp(g(sX)))^{-1} \right)\exp(-sX).
\end{aligned}
$$
Since $N$ is a normal subgroup of $U,$ it  follows that $t-s\in A.$
\end{proof}
\section{Characters and invariant probability measures}
\label{S:CharInvariant}
In this section, we show how a character on an algebraic group over $\QQ$ 
gives rise to an invariant ergodic probability measure on an appropriate adelic solenoid.
\subsection{Reduction to the case $k=\QQ$}
Let $k$ be a number field and $G=\GG(k)$ be the group of  $k$-rational points of a connected linear algebraic group $\GG$ over $k$.
By Weil's restriction of scalars (see \cite[Proposition 6.1.3]{Zimmer}, \cite[6.17-6.21]{Borel-Tits1}),  there is  an algebraic group $\GG'$ over $\QQ$ such that 
$G$ is naturally isomorphic  to the group $G'=\GG'(\QQ)$ of
$\QQ$-points of $\GG'$. If $G=LU$ is a Levi decomposition
of $G$ over $k$, then $G'=L'U'$ is a Levi decomposition of $G'$ over $\QQ,$
where $L'$ and $U'$ are the images of $L$ and $U$ under the isomorphism $G\to G'.$
Moreover,  $G'$ is generated by its unipotent one-parameter subgroups, if  $G$ is generated by unipotent one-parameter subgroups.
These remarks show that it suffices to prove Theorem~\ref{Theo-GenAlgGroup}
in the case $k=\QQ.$

\subsection{Restriction to the unipotent radical}
\label{SS:RestrictionRadical}
Let $G$ be the group of  $\QQ$-rational points of a connected linear algebraic group  over $\QQ$ and let $G=LU$ be a Levi decomposition of $G$.

Let $\psi\in \Char(G).$ Set $\vfi:=\psi|_U.$ By Theorem~\ref{Prop-GNS-InnerAut}, we have
$\vfi\in \Char(U, G)$. So, by Corollary~\ref{Cor-Pro-TracesNilpotent},
$\vfi=\mathcal{F}(\mu)\circ \log$ for a unique $\mu\in \Prob(\widehat{\mathfrak{u}})_{\rm erg}^G,$ where $\mathfrak{u}$ is the Lie algebra of $U.$

We want to determine the set $\Prob(\widehat{\mathfrak{u}})_{\rm erg}^G.$
In the following discussion,  the Lie algebra structure 
of  $\mathfrak{u}$ will play no role, only its linear structure
being relevant. So, we let $E$ be a finite dimensional vector space
over $\QQ$ and recall how the Pontrjagin dual $\widehat{E}$
can be described in terms of  ad\`eles.

Let $\P$ be the set of primes of $\NN.$ Recall that,
for every  $p\in \P ,$ the additive group of the field $\QQ_p$ of $p$-adic numbers is a locally compact  group
containing the subring $\ZZ_p$ of $p$-adic integers as compact open subgroup.
The ring $\AA$ of ad\`eles of $\QQ$  is the restricted product  $\AA= \RRR\times \prod_{p\in \P} (\QQ_p, \ZZ_p)$
relative to the subgroups $\ZZ_p$; thus, 
$$\AA= \left\{(a_\infty, a_2, a_3, \cdots) \in \RRR\times \prod_{p\in \P} \QQ_p\mid a_p \in \ZZ_p \text{ for almost every } p\in \P \right\}.$$
The field $\QQ$ can be viewed as discrete and cocompact subring of the locally compact 
ring $\AA$ via the diagonal embedding
$$
\QQ\to \AA, \qquad q\mapsto (q, q, \dots).
$$

Let $b_1, \dots b_d$ be a  basis of  $E$ over $\QQ.$
Fix a nontrivial unitary character $e$ of $\AA$ which is trivial on  $\QQ.$
For every $a=(a_1, \dots, a_d)\in \AA^d,$ let $ \lambda_a\in \widehat{E}$ be defined  by $$
 \lambda_a(x)= e(\sum_{i=1}^d a_iq_i) \quad \tout x=\sum_{i=1}^d q_i b_i\in E.
 $$
The map $a\mapsto \lambda_a$ factorizes to an  isomorphism of topological groups
 $$
 \AA^d/\QQ^d\to \widehat{E}, \quad a+ \QQ^d\mapsto \lambda_a
 $$
 (see Theorem 3 in Chap.IV, \S 3 of \cite{Weil}).
 So, $ \widehat{E}$ can be identified with the \textbf{adelic solenoid} $\AA^d/\QQ^d.$
 We examine now how this identification behaves under the action of
 $GL(E)$ on $ \widehat{E}.$
 
Set $\QQ_\infty =\RRR.$ Then $GL_d(\QQ)\subset GL_d(\QQ_p)$ 
acts on $\QQ_p^d$ for every $p\in \P\cup \{\infty\}$ in the usual way;
the induced diagonal action of $GL_d(\QQ)$ on $\AA^d$
preserves the lattice $\QQ^d$, giving rise to a (left) action of
$GL_d(\QQ)$ on $\AA^d/\QQ^d$.

 Let $\theta\in GL(E)$ and let $A\in GL_n(\QQ)$ its matrix with respect to the basis $b_1, \dots, b_d.$ One checks that 
 $$\lambda_a \circ \theta= \lambda_{A^t a} \tout a \in \AA^d.$$

 We summarize the previous discussion as follows.
 \begin{proposition}
 \label{Pro-AdelesInvMeasures}
  Let $E$ be a finite dimensional vector space over $\QQ$ of dimension $d.$
The choice  of a basis of $E$ defines an isomorphism of topological groups
$\AA^d/\QQ^d\to \widehat{E}$, which is equivariant for the
action of $GL_d(\QQ)$ given by inverse matrix transpose on $\AA^d/\QQ^d$
and  the dual action of $GL(E)$ on $\widehat{E}.$
This isomorphism  induces a bijection
 $$\Prob(\AA^d/\QQ^d)_{\rm erg}^{G}\to\Prob(\widehat{E})_{\rm erg}^G,$$
 for every subgroup $G$ of $GL(E)\cong GL_d(\QQ).$
  \end{proposition}

 \section{Invariant probability measures and orbit closures on solenoids}
\label{S:InvProba}
For an algebraic $\QQ$-subgroup of $GL_d$
which is generated by unipotent subgroups, we will determine in this section
the invariant probability measures as well as the orbits closures on the adelic solenoid $\AA^d/\QQ^d.$
We have first to treat the case of $S$-adic solenoids.
\subsection{Invariant probability measures and orbit closures on  $S$-adic  solenoids}
 \label{SS:InvariantProba-S-Adic}
 
 Let $\GG$ be an algebraic subgroup of $GL_d$  defined over $\QQ.$ 
For every subring $R$ of  an overfield of $\QQ,$ we denote by  
$\GG(R)$  the group of elements of $\GG$ with coefficients in $R$ and determinant invertible in $R.$ In particular, $\GG(\QQ)=\GG\cap GL_d(\QQ).$

 Fix an integer $d\geq 1$ and let $S$ be a finite subset of $\P \cup \{\infty\}$
with $\infty \in S.$ Set  
$$\QQ_S^d:= \prod_{p\in S} \QQ_p^d$$ and let 
 $\ZZ[1/S]$  denote the subring  of $\QQ$ generated by 
 $1$ and $(1/p)_{p\in S\cap \P}.$
Then $\ZZ[1/S]^d$ embeds diagonally as a cocompact discrete subring 
of $\QQ_S^d.$

The  product group 
$$
\GG(\QQ_S):= \prod_{p\in S}\GG(\QQ_p)
$$
is a locally compact group and acts on $\QQ_S^d$ in the obvious way. 
The group $\GG( \ZZ[1/S])$  embeds diagonally as discrete subgroup of $\GG(\QQ_S).$
As $\GG( \ZZ[1/S])$  preserves $\ZZ[1/S]^d,$  this gives rise to  an action 
of $\GG( \ZZ[1/S])$ on   the  \textbf{$S$-adic solenoid}
$$X_S:=\QQ_S^d/ \ZZ[1/S]^d,$$
which is a compact connected abelian group.

A \textbf{unipotent one-parameter subgroup} of $\GG(\QQ_S)$
is a subgroup of $\GG(\QQ_S)$ of the form 
$\{(u_p(t_p))_{p\in S} \mid t_p\in \QQ_p, p\in S\}$ for 
$\QQ$-rational homomorphisms $u_p: \GG_a\to \GG$  from the
additive group $\GG_a$ of dimension 1 to $\GG.$

We aim to describe the $\GG(\ZZ[1/S])$-invariant probability measures 
on $X_S$ as well as orbit closures of points in $X_S$.
  Our results will be deduced from Ratner's measure rigidity and topological rigidity theorems in the $S$-adic setting (see \cite{Ratner2} and  \cite{Margulis-Tomanov}); actually, we will need the more precise version of  Ratner's results 
 in the $S$-arithmetic case from \cite{Tomanov}.

\subsubsection{Invariant probability measures}
For a closed subgroup $Y$ of $X_S$ and for $x\in X,$
we denote by $\mu_{x+Y}\in \Prob(X_S)$ the image of 
the normalized Haar $\mu_{Y}$
under the map $X_S\to X_S$ given by translation by $x.$

Let $V$ be a linear subspace of $\QQ^d.$
Denote by  $V(\QQ_p)$ the linear span of $V$ in 
$\QQ_p^d$ for $p\in S$. Then $V(\QQ_S):=  \prod_{p\in S} V(\QQ_p)$ 
is a subring of $\QQ_S^d$ and $V(\ZZ[1/S]):= V\cap \ZZ[1/S]^d$
is a cocompact lattice in $V(\QQ_S)$. So, $V(\QQ_S)/V(\ZZ[1/S])$ is a 
\textbf{subsolenoid} of $X_S$, that is,  a closed and connected subgroup of $X_S.$

 \begin{proposition}
 \label{Prop-InvProba-S-adic}
 Assume that $\GG(\QQ_S)$ is generated by unipotent one-parameter unipotent 
 subgroups. Let $\mu$ be an ergodic $\GG(\ZZ[1/S])$-invariant probability measure
 on the Borel subsets of $X_S.$ There exists a pair $(a,V)$  consisting of a point $a\in \QQ_S^d$ and a  $\GG(\QQ)$-invariant linear subspace $V$ of $\QQ^d$ 
 with the following properties:
 \begin{itemize}
 \item[(i)] $g(a)\in a + V(\QQ_S)$ for every  $g\in \GG(\QQ_S);$
 \item[(ii)] $\mu= \mu_{x+Y}$, where $x$ and $Y$ are  the images of $a$ and $V(\QQ_S)$
in $X_S.$
 \end{itemize}
 \end{proposition}
 
 \begin{proof}
 We consider the semi-direct product
 $$\widetilde{G}:=\GG(\QQ_S)\ltimes \QQ_S^d,$$ given by the
natural action of $\GG(\QQ_S)$ on $\QQ_S^d$. Then $\widetilde{G}$ is a locally compact  group containing 
$$\widetilde{\Ga}:=\GG(\ZZ[1/S])\ltimes \ZZ[1/S]^d$$ 
as discrete subgroup. Since $\GG(\QQ_S)$ is generated by unipotent
one-parameter subgroups, there is no
 non-trivial morphism $\GG\to GL_1$ defined
 over $\QQ$. It follows (see \cite[Theorem 5.6]{Borel}) that $\Ga:=\GG(\ZZ[1/S])$ has finite covolume in $\GG(\QQ_S)$  and so $\widetilde{\Ga}$ is an $S$-arithmetic lattice
 in $\widetilde{G}.$
 
We now use the ``suspension technique'' from  \cite{Witte} to obtain an ergodic  $\GG(\QQ_S)$-invariant probability measure $\widetilde{\mu}$ on $\widetilde{G}/\widetilde{\Ga}.$
Specifically, we embed $X_S$ as subset of   $\widetilde{G}/\widetilde{\Ga}$ in the obvious way. Observe that 
the action of $\GG(\ZZ[1/S])$ by automorphisms on $X_S$ becomes the  
action of $\GG(\ZZ[1/S])$ by translations on  $\widetilde{G}/\widetilde{\Ga}$
under this embedding.

View $\mu$ as a $\GG(\ZZ[1/S]))$-invariant probability measure on $\widetilde{G}/\widetilde{\Ga}$ which is supported on the image of $X_S$.
Let $\widetilde{\mu}$ be the probability measure on $\widetilde{G}/\widetilde{\Ga}$  defined by 
 $$\widetilde{\mu}= \int_{\GG(\QQ_S)/\Ga} t_g(\mu) d\nu(g \Ga),$$
where $\nu$ be the unique  $\GG(\QQ_S)$-invariant  probability 
measure on $\GG(\QQ_S)/\Ga$
and $t_g(\mu)$ denotes the image of $\mu$ under the translation by $g.$
Then $\widetilde{\mu}$ is $\GG(\QQ_S)$-invariant and is ergodic under this action.

By the refinement \cite[Theorem 2]{Tomanov} of Ratner's theorem,
there exists   a $\QQ$-algebraic subgroup $\LL$ of $\GG$, an $\LL(\QQ)$-invariant
vector subspace $V$ of $\QQ^d$, a finite index subgroup $H$ of 
$\LL(\QQ_S) \ltimes  V(\QQ_S)$, and an element $g\in  \widetilde{G}$ with the following properties:
\begin{itemize}
\item $\GG(\QQ_S) \subset H^g:=g H g^{-1};$
\item $H\cap \widetilde{\Ga} $ is a lattice in $H;$ 
\item $\widetilde{\mu}$ is the unique $H^g$-invariant probability measure on  $\widetilde{G}/\widetilde{\Ga}$  supported  on $gH\widetilde{\Ga}/ \widetilde{\Ga}=H^g g\widetilde{\Ga}/ \widetilde{\Ga}$.
\end{itemize}
Since $\widetilde{G}=\GG(\QQ_S)\ltimes \QQ_S^d,$ there exists
$g'\in  \GG(\QQ_S)$ such that $a:=g'g$ belongs to $\QQ_S^d.$  Then 
$\GG(\QQ_S) \subset H^a$ and, since $\widetilde{\mu}$ is $\GG(\QQ_S)$-invariant,  
 $\widetilde{\mu}$ coincides with the $H^a$-invariant probability measure   supported  on 
 $H^a a\widetilde{\Ga}/ \widetilde{\Ga}$. As a result, we may assume  above that 
 $g=a\in \QQ_S^d.$

 The image $t_{a^{-1}}(\widetilde{\mu})$  of $\widetilde{\mu}$ under the translation by $a^{-1}$  coincides with the unique $H$-invariant probability measure  
 on  $\widetilde{G}/\widetilde{\Ga}$   supported  on 
 $H\widetilde{\Ga}/ \widetilde{\Ga}.$ 
Observe that $\GG(\QQ_S) \subset H^a$ implies that 
 $$
g(a)-a\in V(\QQ_S) \quad\text{for every} \quad g\in \GG(\QQ_S),
 $$
 where we write $g(a)$ for $gag^{-1}.$
 
Let 
$$p:\widetilde{G}/\widetilde{\Ga}\to \GG(\QQ_S)/\Ga$$ 
be the natural  $\GG(\QQ_S)$-equivariant  map.
We have  $$
t_{a^{-1}}(\widetilde{\mu})= \int_{\GG(\QQ_S)/\Ga} t_{a^{-1}}\left(t_g(\mu)\right)
d\nu(g\Ga). \leqno{(*)} 
 $$
 and $t_{a^{-1}}(t_g(\mu))\left(p^{-1} (g\Ga/\Ga)\right)=1$ for every 
 $g\in \GG(\QQ_S).$
 So, the formula $(*)$ provides a decomposition of $t_{a^{-1}}(\widetilde{\mu})$
as integral over $\GG(\QQ_S)/\GG(\ZZ[1/S])$ of probability measures supported on the fibers of $p.$  

Now, knowing that $t_{a^{-1}}(\widetilde{\mu})$ is the  $H$-invariant probability measure  
supported on  $H\widetilde{\Ga}/ \widetilde{\Ga},$ we can perform a second 
such decomposition of $t_{a^{-1}}(\widetilde{\mu})$ over $\GG(\QQ_S)/\Ga$.
The measures supported on the fibers of $p$ in this last decomposition are translates of the normalized Haar measure $\mu_Y$ of the image $Y$ of  $H\cap \QQ_S^d$ 
in $X_S\cong \QQ_S^d\widetilde{\Ga}/ \widetilde{\Ga}.$
By uniqueness, it follows that  $t_{a^{-1}}(\mu)=\mu_Y$,  that is, $\mu=\mu_{x+Y}$, where 
$a$ is the image of $x$ in $X_S$
(for more details, see the proof of Corollary 5.8 in \cite{Witte}).

 To finish the proof, observe that, since $V(\QQ_S)$ is divisible, it has  no proper subgroup of finite index and so $H\cap \QQ_S^d=H\cap  V(\QQ_S)= V(\QQ_S)$. 
 \end{proof}
The    pairs $(x,Y)$ as in Proposition~\ref{Prop-InvProba-S-adic}
for which $\mu_{x+Y}$ is ergodic are characterized by  the following general result.

For a compact group $X,$ we denote by $\Aut(X)$ the group of
continuous automorphisms of $X$ and by   $\Aff(X)=\Aut(X)\ltimes X$
the group of affine transformations of $X$.

\begin{proposition}
\label{Prop-ErgodicityAffineSubspace}
 Let $G$  be a  countable group, $X$ be a  compact  abelian group 
 and $\alpha:G\to \Aut(X)$ an action of $G$ by automorphisms
 of $X.$  Let $x_0\in X$ and 
 let $Y$ be a connected closed subgroup  of $X$ such $x_0+Y$ is $G$-invariant.
 Then $Y$ is $G$-invariant and $\alpha_g(x_0)-x_0\in Y$ for every $g\in G.$
 Moreover, the following properties are equivalent:
 \begin{itemize}
 \item[(i)] $\mu_{x_0+Y}$ is not ergodic under the restriction of the  $G$-action to $x_0+Y;$
 \item[(ii)] there exists a proper closed  connected subgroup $Z$ of $Y$
 and a finite index subgroup $H$ of $G$ such that $\alpha_h(x)-x\in Z$ for every $x\in x_0+Y$ and $h\in H;$
 \item[(iii)]  for every $x\in x_0+Y$, the set  $\{\alpha_g(x)-x\mid g\in G\}$ is not dense
 in $Y;$ 
 \item[(iv)]  there exists a subset $A$ of $x_0+Y$ with $\mu_{x_0+Y}(A)>0$ such that
  $\{\alpha_g(x)-x\mid g\in G\}$ is not dense in $Y$ for every $x\in A.$
 \end{itemize} 
 \end{proposition}
 \begin{proof}  
 The fact that  $Y$ is $G$-invariant and that $\alpha_g(x)-x\in Y$ for every $g\in G$ is obvious.
 
 The homeomorphism $t:x_0+Y\to Y$ given by the translation by $-x_0$
intertwines the  action $\alpha$ of $G$ on $x_0+Y$ 
with the action $\beta:G\to \Aff(Y)$ by affine transformations of $Y$,
given  by 
$$\beta_g(y) =\alpha_g(y)+ \alpha_g(x_0)-x_0\tout g\in G, y\in Y.$$ 
Moreover, the image of $\mu_{x_0+Y}$ under $t$ is the Haar measure $\mu_Y$ on $Y.$

Assume that the action $\beta$ is not ergodic. Then there exists 
a proper closed connected subgroup $Z$ of $Y$ 
which invariant under the action $\alpha$ of $G$ and a finite index subgroup $H$ of $G$ 
such that the image of $H$ in $\Aff(Y/Z)$, for the action induced by $\beta$, is trivial
(see \cite[Proposition 1]{BekkaFrancini}). 
This means that $\alpha_h(x)-x\in Z$ for every $x\in x_0+Y$ and $h\in H.$ 
So, (i) implies (ii). 

Assume that (ii) holds and let $x\in x_0+Y.$ Then the image of 
the set $\{\alpha_g(x)- x\mid g\in G\}$ in $Y/Z$ is  finite.
However, since $Y$ is connected, $Y/Z$ is infinite
and so   $\{\alpha_g(x)- x\mid g\in G\}$ is not dense in $Y.$ So, (ii) implies (iii).

The fact that (iii) implies (iv) is obvious.
Assume that $\beta$ is ergodic. Since the support of $\mu_Y$ is
$Y,$  for $\mu_Y$-almost every $y\in Y,$ the $\beta(G)$-orbit of $y$  is dense
in $Y,$  that is,
$\{\alpha_g(x)- x\mid g\in G\}$ is dense in $X$ for $\mu_{x_0+Y}$-almost every 
$x\in x_0+Y.$ So, (iv) implies (i).
 \end{proof}

 We  will need  a description of the $\GG(\ZZ[1/S])$-invariant
 (not necessarily ergodic) probability measures on $X_S.$ 
 For this,   we adapt for our situation  some ideas from \cite[\S 2]{Mozes-Shah},
 where  such description was given in the context of real Lie groups.

 Let  $\vfi: \QQ_S^d\to X_S= \QQ_S^d/\ZZ[1/S]^d$ denote the canonical projection.
 Observe that, if  $V$ is a $\GG(\QQ)$-invariant  linear subspace  of $\QQ^d$
and if  $a\in  \QQ_S^d$ is such that $g(a)\in a + V(\QQ_S)$ for all $ g\in \GG(\ZZ[1/S]),$
then $\vfi(a+ V(\QQ_S))$ is a closed and $\GG(\ZZ[1/S])$-invariant   subset of $X_S$

Denote by $\mathcal{H}$ the set  of  $\GG(\QQ)$-invariant  linear subspaces of $\QQ^d.$   For $V\in \mathcal{H}$, define $\mathcal{N}(V, S)\subset  \QQ_S^d$ to be the 
set of $a\in \QQ_S^d$ with the following properties:
\begin{itemize}
\item $g(a)\in a + V(\QQ_S)$ for every  $g\in \GG(\ZZ[1/S])$ and 
\item $\vfi(\{ g(a)-a\mid g\in \GG(\ZZ[1/S])\})$ is dense in $\vfi(V(\QQ_S)).$
\end{itemize}

 \begin{lemma}
 \label{Lemma-GenericSingular}
  For $V, W\in \mathcal{H}$, we have $\vfi(\mathcal{N}(V, S))\cap \vfi(\mathcal{N}(W, S))\neq \emptyset$
 if and only if $V=W.$
 \end{lemma}
 \begin{proof}
 Assume that $\vfi(\mathcal{N}(V, S))\cap \vfi(\mathcal{N}(W, S))\neq \emptyset$. So, there exist $a\in \mathcal{N}(V, S)$ and $b\in \ZZ[1/S]^d$ such that $a+b\in \mathcal{N}(W, S).$ It follows that 
 $\vfi(\{ g(a)-a\mid g\in \GG(\ZZ[1/S])\})$ is a dense subset of $\vfi(V(\QQ_S))$ and of $\vfi(W(\QQ_S)).$
 Hence, $\vfi(V(\QQ_S))=\vfi(W(\QQ_S))$.  This implies that the $\RRR$-vector space
 $V(\RRR)$ is contained in $W(\RRR)+\ZZ[1/S]^d$ and so in $W(\RRR)$, by connectedness.
 Hence, $V\subset W.$ Similarly, we have $W\subset V$.
  \end{proof}
 
 We can now give  a  description of the finite $\GG(\ZZ[1/S])$-invariant measures on $X_S.$ 
  \begin{proposition}
 \label{Prop-InvProba-S-adic-bis}
 Assume that $\GG(\QQ_S)$ is generated by unipotent one-parameter 
 subgroups and let $\mu\in \Prob(X_S)$ be a $\GG(\ZZ[1/S])$-invariant probability measure
 on $X_S.$  
 For $V\in \mathcal{H},$  denote by   $\mu_V$   the 
 restriction of $\mu$ to $\vfi(\mathcal{N}(V, S)).$
 \begin{itemize}
 \item[(i)] We have  
 $$\mu\left(\bigcup_{V\in\mathcal{H}}\vfi(\mathcal{N}(V, S))\right)=1;$$
  moreover, $\mu_V(\vfi(\mathcal{N}(V', S))=0$ for all $V,V'\in \H$ with $V\neq V';$
  so, we have a  decomposition 
 $$
 \mu= \bigoplus_{V\in \mathcal{H}} \mu_V.
 $$
 \item[(ii)] Let  $V\in \mathcal{H}$ be such that $\mu_V\neq 0.$ Then $\mu_V$ is $\GG(\ZZ[1/S])$-invariant.
 Moreover, if 
 $$\mu_V=\int_\Omega \nu_{V,\omega} d\omega$$
  is a decomposition of $\mu_V$ into
 ergodic $\GG(\ZZ[1/S])$-invariant  components $\nu_{V,\omega}$, then, 
 for every $\omega\in \Omega,$ we have $\nu_{V,\omega}= \mu_{x_\omega+Y},$ 
  where $Y=\vfi(V(\QQ_S))$ and $x_\omega=\vfi(a_\omega)$ for some   $a_\omega\in \mathcal{N}(V, S).$
 \end{itemize}

 \end{proposition}
 \begin{proof}
 Let 
 $$\mu=\int_\Omega \nu_{\omega} d\omega$$
be  a decomposition of $\mu$ into
  $\GG(\ZZ[1/S])$-invariant  ergodic probability measures $\nu_{\omega}$
 on $X_S.$
 
 Fix $\omega\in \Omega$. By Proposition~\ref{Prop-InvProba-S-adic},
 there exists  $a_\omega \in \QQ_S^d$ and $V_\omega\in\H$ such that 
  $g(a_\omega)\in a_\omega + V_\omega(\QQ_S)$ for every  $g\in \GG(\QQ_S)$ 
  and $\nu_\omega= \mu_{x_\omega+Y_\omega}$, where  $Y_\omega=\vfi(V_\omega(\QQ_S))$ and $x_\omega=\vfi(a_\omega)$.
  Since $\mu_{x_\omega+Y_\omega}$ is ergodic, there exists
  a subset $A_\omega$ of  $x_\omega+Y_\omega$ with $\mu_{x_\omega+Y_\omega}(A_\omega)=1$ such that  the $\GG(\ZZ[1/S])$-orbit of $x$
  is dense in $x_\omega+Y_\omega$ for every $x\in A_\omega$
  (see Proposition~\ref{Prop-ErgodicityAffineSubspace}). It is clear that
  $x\in  \vfi(\mathcal{N}(V_\omega, S))$ for every $x\in A_\omega.$ 
   It follows that 
  $\nu_\omega (\vfi(\mathcal{N}(V_\omega, S))=1$ for every $\omega\in \Omega$
  and hence $\mu(\bigcup_{V\in\mathcal{H}}\vfi(\mathcal{N}(V, S))=1.$
 Since the measurable subsets $\vfi(\mathcal{N}(V, S))$ of $X_S$ are mutually disjoint
 (Lemma~\ref{Lemma-GenericSingular}) and since $\mathcal{H}$ is countable,
 we have a direct sum decomposition
 $$
 \mu=\bigoplus_{V\in \H} \int_{\omega: V_\omega=V}  \nu_{\omega} d\omega
$$
and $\mu_V= \int_{\omega: V_\omega=V}  \nu_{\omega} d\omega$
with $\nu_\omega= \mu_{x_\omega +Y},$ where $Y=V(\QQ_S)$ and $x_\omega=\vfi(a_\omega)$ for some $a_\omega\in \mathcal{N}(V, S).$

 \end{proof} 
  We will later need the following lemma.
 \begin{lemma}
 \label{Lemma-NonErgInvMeasure}
 Let $V$ be a $\GG(\QQ)$-invariant  linear subspace  of $\QQ^d$
 and $S$ a finite subset  of $\P\cup \{\infty\}.$ There exists 
 exists a linear subspace $W^S$  of $\QQ^d$ containing $V$  such that 
  $$
  \left \{a \in \QQ^d\mid g(a)\in a + V(\QQ_S) \quad \text{for all}\quad g\in \GG(\ZZ[1/S])\right\}=W^S(\QQ_S).$$ 
   \end{lemma}
 \begin{proof}
 Choose a complement $V_0$ of $V$ in $\QQ^d$ and let 
 $\pi_0:  \QQ^d\to V_0$ be the 
 corresponding projection. Let $p\in \P\cup \{\infty\}.$ Then 
 $$\QQ_p^d=V(\QQ_p)\oplus V_0(\QQ_p)$$
  and  the linear  extension of $\pi_0$, again denoted by $\pi_0$, is the corresponding projection  $\QQ_p^d\to V_0(\QQ_p)$. 
 
 Let  $g\in GL_d(\QQ)$. Denote by $W_g\subset \QQ^d$ the kernel of $\pi_0\circ (g-I_{\QQ^d}).$
 For $a\in \QQ_p^d,$ we  have 
 $$g(a)\in a+V(\QQ_p) \Longleftrightarrow a\in \ker \left(\pi_0\circ (g-I_{\QQ_p^d})\right).$$
 So, $\ker (\pi_0\circ (g-I_{\QQ_p^d}))= W_g(\QQ_p).$  
 The linear subspace 
 $$W^S:=\bigcap_{g\in \GG(\ZZ[1/S])} W_g$$
  of $\QQ^d$  has the required property. 
 \end{proof}

\subsubsection{Orbit closures}

We now turn to the description of orbit closures of points in $X_S.$
Recall that $\vfi: \QQ_S^d\to X_S$ denotes the canonical projection.
\begin{proposition}
 \label{Prop-OrbitClosure-S-adic}
 Assume that $\GG(\QQ_S)$ is generated by unipotent one-parameter  subgroups. Let $a\in \QQ_S^d$ and $x=\vfi(a)\in X_S.$
 There exists a  $\GG(\QQ)$-invariant linear subspace $V$ of $\QQ^d$ 
 with the following properties:
 \begin{itemize}
 \item[(i)] $g(a)\in a + V(\QQ_S)$ for every  $g\in \GG(\QQ_S);$
 \item[(ii)] the closure of the $\GG(\ZZ[1/S])$-orbit of $x$   in $X_S$ coincides with $x+\vfi(V(\QQ_S)).$
 \end{itemize}
 \end{proposition}
 \begin{proof}
 As in the proof of  Proposition~\ref{Prop-InvProba-S-adic}, we  consider the semi-direct product  $\widetilde{G}=\GG(\QQ_S)\ltimes \QQ_S^d$ 
 and embed $X_S$ as closed subset of    $\widetilde{G}/\widetilde{\Ga}$, where
$\widetilde{\Ga}=\GG(\ZZ[1/S])\ltimes \ZZ[1/S]^d.$

 By the refinement \cite[Theorem 1]{Tomanov} of Ratner's theorem about orbit closures,
there exists   a $\QQ$-algebraic subgroup $L$ of $G$, an $L(\QQ)$-invariant
vector subspace $V$ of $\QQ^d$ and  a finite index subgroup $H$ of 
$L(\QQ_S) \ltimes  V(\QQ_S)$ with the following properties:
\begin{itemize}
\item $\GG(\QQ_S) \subset H^a:=a H a^{-1};$
\item $H\cap\widetilde{\Ga}$ is a lattice in $H;$ 
\item the closure $\overline{\GG(\QQ_S)x}$ of  the $\GG(\QQ_S)$-orbit of $x$ is $H^ax$, that is, 
$aH\widetilde{\Ga}/\widetilde{\Ga}.$
\end{itemize}

We claim that 
$$
\overline{\GG(\ZZ[1/S])x}=\overline{\GG(\QQ_S)x}\cap X_S.
$$
We only have to show that $\overline{\GG(\QQ_S)x}\cap X_S$
is contained in  $\overline{\GG(\ZZ[1/S])x},$ the reverse inclusion being obvious.

Set $\Gamma=\GG(\ZZ[1/S])$ and $\Lambda=\ZZ[1/S]^d.$
Choose a fundamental domain $\Omega\subset \GG(\QQ_S)$ for $\GG(\QQ_S)/\Gamma$
which is a neighbouhood of $e$ and  a compact fundamental domain  $K\subset \QQ_S^d$  for $\QQ_S^d/\La$

Let $y \in \overline{\GG(\QQ_S)x}\cap X_S.$
Then there exists a sequence $g_n \in \GG(\QQ_S)$ such that $\lim_n (g_n,e) x=y$. Write $g_n= \omega_n \ga_n$ for
$\omega_n\in \Omega$ and $\ga_n\in \Ga$ and $\ga_n(a)= k_n +\la_n$ for 
$k_n\in K$ and $\la_n\in \La.$
Then 
$$
\begin{aligned}
y&=\lim_n (g_n,e)x=  \lim_n(\omega_n,e) (\ga_n, e) (e,a)\widetilde{\Ga}=\lim_n(\omega_n,e) (\ga_n, \ga_n(a))\widetilde{\Ga}\\
&= \lim_n(\omega_n,e) (e, k_n) \widetilde{\Ga}= \lim_n(\omega_n, \omega_n(k_n))\widetilde{\Ga}.\\
\end{aligned}
$$
On the one hand, it follows that $\lim_n\omega_n \delta_n= e$  for some $\delta_n\in \Ga.$
So, for large $n,$ we have $\omega_n \delta_n\in \Omega$ and, since $\omega_n\in\Omega$,
we have $\delta_n=e,$ that is, $\lim_n \omega_n=e.$
On the other hand, as $K$ is compact, we can assume that $\lim_nk_n=k\in K$
exists. Therefore, we have $\lim_n (\omega_n, \omega_n(k_n)) = (e,k)$ and so 
$y= k+\La$  and 
$$
y= \lim_n (k_n+ \La)= \lim_n (\ga_n(a)+\La),
 $$
that is, $y\in \overline{\GG(\ZZ[1/S])x}.$ So, the claim is proved.

We have 
$$\left(aH\widetilde{\Ga}/\widetilde{\Ga}\right) \cap X_S= \vfi(a+H\cap  V(\QQ_S))$$
and, since (as in the proof of  Proposition~\ref{Prop-InvProba-S-adic}) 
$H\cap  V(\QQ_S)= V(\QQ_S)$, this finishes the proof.

 \end{proof}

\subsection{Invariant probability measures and orbit closures  on  adelic  solenoids}
 \label{SS:InvariantProba-S-Adelic}
Let $\GG$ be an algebraic subgroup of $GL_d$  defined over $\QQ.$
 We are now ready to deal with the description 
 of the  $\GG(\QQ)$-invariant probability measures and the orbit closures for
the adelic solenoid $$X:=\AA^d/\QQ^d.$$

Denote by $\mathcal{S}$ the set of finite subsets $S$ 
of $\P\cup\{\infty\}$ with $\infty \in S.$

Let $S\in  \mathcal{S}$. 
It is well-known (see \cite{Weil}) that
$$\AA^d= \left(\QQ_S^d \times \prod_{p\notin \P} \ZZ_p^d\right) +\QQ^d$$
and that 
$$\left(\QQ_S^d \times \prod_{p\notin \P} \ZZ_p^d\right)\cap \QQ^d= \ZZ[1/S]^d.$$ 
 This gives rise to a well defined projection  
$$\pi_S: X\to X_S=\QQ_S^d/\ZZ[1/S]^d$$ given by 
$$\pi_S\left((a_S, (a_p)_{p\notin S}) +\QQ^d\right)= a_S+\ZZ[1/S]^d \tout a_S\in \QQ_S^d, (a_p)_{p\notin S}\in \prod_{p\notin \P} \ZZ_p^d.$$
So, the fiber of $\pi_S$ over  a point $a_S+\ZZ[1/S]^d\in X_S$ is
 $$\pi_S^{-1}(a_S+\ZZ[1/S]^d)= \{(a_S, (a_p)_{p\notin S}) +\QQ^d\mid a_p\in \ZZ_p^d \text{ for all } p\notin S\}.$$
Observe that $\pi_S$ is $GL_d(\ZZ[1/S])$-equivariant.

Let $S'\in  \mathcal{S}$ with $S\subset S'.$ 
Then
$$\AA_{S'}^d= \left(\QQ_S^d \times \prod_{p\in S'\setminus S} \ZZ_p^d\right) +\ZZ[1/S']^d,$$
$$\left(\QQ_S^d \times \prod_{p\in S'\setminus S} \ZZ_p^d\right)\cap \ZZ[1/S']^d= \ZZ[1/S]^d,$$ 
and we have a similarly defined $GL_d(\ZZ[1/S])$-equivariant projection 
$\pi_{S', S}: X_{S'}\to X_S.$ Observe that $\pi_S= \pi_{S', S}\circ \pi_{S'}.$

Let $V$ be a linear subspace of $\QQ^d$. For $p\in \P$,
we write  $V(\ZZ_p)$ for the $\ZZ_p$-span of 
$V(\ZZ)$ in $V(\QQ_p)$; the ad\`ele space corresponding to $V$ is
$$
V(\AA)=\bigcup_{S\in  \mathcal{S}} \left(V(\QQ_S) \times  \prod_{p\notin S} V(\ZZ_p)\right).
$$
We denote by  $\vfi$ the canonical projection $\AA^d\to X.$
The image of $\vfi(V(\AA))$ in $X$ can be written as
$$\vfi(V(\AA))= \vfi\left(V(\RRR) \times  \prod_{p\in \P} V(\ZZ_p)\right)$$
and is a closed and connected subgroup of $X$.
Conversely, every closed and connected subgroup of $X$ is of the 
form $\vfi(V(\AA))$ for a unique linear subspace $V$ of $\QQ^d$ (see Lemma~\ref{Lem-Subsolenoids} below).

The following simple fact will be useful.

\begin{lemma}
\label{lem-Intersection}
Let $V$ be a linear subspace of $\QQ^d$.
Set 
$$\Omega:=\bigcap_{S}\vfi\left(V(\RRR) \times  \prod_{p\in S} V(\ZZ_p) \times \prod_{p\notin S} \ZZ_p^d\right) \subset X.
$$
where $S$ runs over the finite subsets of $\P.$
Then 
$$\Omega=\vfi(V(\AA)).$$
\end{lemma}
\begin{proof}
 It is clear that $\vfi(V(\AA))$ is contained in  $\Omega.$
 Conversely, let $x\in \Omega.$ Then there exists $a=(a_p)_{p\in \P\cup\{\infty\}}\in \AA^d$
 with  $\vfi(a)=x$ such that $a_{\infty}\in V(\RRR)$ and $a_p\in \ZZ_p^d$ for all $p\in \P$.
 We claim that $a_p\in V(\ZZ_p)$ for all $p\in \P$. 
 
 Indeed, let $p_0\in \P$. There exists $q\in \QQ^d$ such that 
 $(a_p+q)_{p\in \P\cup\{\infty\}} \in \RRR^d\times   \prod_{p\in \P} \ZZ_p^d$ with 
 $a_{\infty}+q\in V(\RRR)$ and  $a_{p_0}+q\in V(\ZZ_{p_0})$.
 For every $p\in \P,$ we have 
 $q=(a_{p}+q)-a_{p} \in \ZZ_p^d$ and hence $q\in \ZZ^d.$
 Since $a_{\infty}\in V(\RRR),$ we also have 
 $q=(a_{\infty}+q)-a_{\infty}\in V(\RRR)$. Hence, $q\in V(\ZZ)\subset V(\ZZ_{p_0})$ and therefore
 $a_{p_0}= (a_{p_0}+q)-q\in V(\ZZ_{p_0}).$

\end{proof}
\subsubsection{Invariant probability measures}

We will denote by  $\vfi$ the canonical projection $\AA^d\to X$ 
and by $\vfi_S$  the projection  $\QQ_S^d\to X_S$ for a set  $S\in  \mathcal{S}.$

\begin{theorem}
\label{Theo-ErgodicityAdelic}
Let $\GG$ be a connected algebraic subgroup of $GL_d$  defined over $\QQ.$
Assume that $\GG(\QQ)$ is generated by unipotent  one-parameter 
 subgroups. Let $\mu$ be an ergodic $\GG(\QQ)$-invariant probability measure
 on the Borel subsets of $X=\AA^d/\QQ^d.$ There exists a pair $(a,V_0)$  consisting of a point $a\in \AA^d$ and a  $\GG(\QQ)$-invariant linear subspace $V_0$ of $\QQ^d$ 
 such that 
 $$\mu= \mu_{x+Y},$$
 for $x=\vfi(a)$ and $Y=\vfi(V_0(\QQ))$.
 Moreover, $a$ can be chosen so that the set $\{g(a)-a\mid g\in \GG(\QQ)\}$ is dense in $V_0(\AA).$
 \end{theorem}
\begin{proof}
For $S\in \mathcal{S},$
let $\mu_S$ be the image of $\mu$ under the projection 
$$\pi_S:X\to X_S.$$ 
Then $\mu_S$ is  a $\GG(\ZZ[1/S]$-invariant probability measure on $X_S$ and,
by Proposition~\ref{Prop-InvProba-S-adic-bis},  we have a  decomposition 
 $$
 \mu_S= \bigoplus_{V\in \mathcal{H}} \mu_{S,V}
 $$
 with mutually singular   measures $\mu_{S,V}$ on $X_S$ such that
 $$\mu_{S,V}(X_S \setminus \vfi_S(\mathcal{N}(V,S))=0.$$
  
  Fix $S_0\in \mathcal{S}$  and $V_0\in \mathcal{H}$ with 
$\mu_{S_0,V_0}\neq 0$ and such   that
 $$
 \dim V_0=\max \left\{\dim V\mid  \mu_{S, V}\neq 0 \text{ for some } S\in \mathcal{S}\right\}.
 $$
 
  Write 
  $$\P\cup\{\infty\}=\cup_{n\geq 0} S_n$$
    for an increasing  sequence of subsets $S_n \in \mathcal{S}$.
Denote by $\mu_n$ instead of $\mu_{S_n}$
 the image of $\mu$ under the projection  $\pi_{S_n}:X\to X_{S_n}$.
 Set 
 $$c:= \mu_{0}(\vfi_{S_0} (\mathcal{N}(V_0,S_0))>0.$$

\vskip.2cm
$\bullet$ {\it First step.}  We claim that 
$$\mu_n (\vfi_{S_n}(\mathcal{N}(V_0,S_n))\geq c \tout n\geq 1.$$
Indeed,  let  $V\in \H$ be such that  $\mu_{S_n,V}\neq 0.$
Recall that 
$$\pi_{S_n, S_0}: X_{S_n}\to X_{S_0}$$
is the natural $\GG(\ZZ[1/S_0])$-equivariant projection.
Let 
$$x\in  \pi_{S_n,S_0}^{-1} (\vfi_{S_0}(\mathcal{N}(V_0, S_0)))\cap \vfi_{S_n}(\mathcal{N}(V, S_n)).$$
Then, on the one hand, $x_0:= \pi_{S_n, S_0}(x)\in \vfi_{S_0}(\mathcal{N}(V_0, S_0))$
and hence the set $\{g(x_0)-x_0\mid g\in \GG(\ZZ[1/S_0])\}$ is dense in 
$\vfi_{S_0}(V_0(\QQ_{S_0}))$.  On the other hand,
since $x\in \vfi_{S_n}(\mathcal{N}(V, S_n))$ and since $\GG(\ZZ[1/S_0])$
is contained in $\GG(\ZZ[1/S_n]),$ the set  $\{g(x)-x\mid g\in \GG(\ZZ[1/S_0])\}$ is contained in  $\vfi_{S_n}(V(\QQ_{S_n}))$.
As 
$$\pi_{S_n, S_0}(\vfi_{S_n}(V(\QQ_{S_n}))= \vfi_{S_0}(V(\QQ_{S_0}))$$
and as $\pi_{S_n, S_0}$ is $\GG(\ZZ[1/S_0])$-equivariant and continuous,
it follows that 
$$\vfi_{S_0}(V_0(\QQ_{S_0})) \subset  \vfi_{S_0}(V(\QQ_{S_0})).$$
This implies that  $V_0\subset V$ (see  the proof of Lemma~\ref{Lemma-GenericSingular}).
It follows that $V=V_0,$ by maximality of the dimension of $V_0.$
 This shows that 
 $$
\mu_n\left( \pi_{S_n,S_0}^{-1} (\vfi_{S_0}(\mathcal{N}(V_0,S_0))) \cap\vfi_{S_n}(\mathcal{N}(V,S_n)) \right)=0
\quad \text{for every} \quad V\neq V_0$$
and hence that
$$
\mu_n \left (\pi_{S_n,S_0}^{-1} (\vfi_{S_0}(\mathcal{N}(V_0,S_0)))\right)\leq  \mu_n(\vfi_{S_n}(\mathcal{N}(V_0,S_n))).
$$
Since $\mu_0=\mu_{S_0}$ is the image of 
$\mu_n$ under  $\pi_{S_n,S_{0}}$, we have 
 $$\mu_n \left (\pi_{S_n,S_0}^{-1} (\vfi_{S_0}(\mathcal{N}(V_0,S_0)))\right)=
 \mu_{0}(\vfi_{S_0} (\mathcal{N}(V_0,S_0)),$$
 and the  claim is proved.
\vskip.2cm

For every $n\geq 0,$ let $W^n= W^{S_n}$ be the  linear subspace of $\QQ^d$  
defined by $V_0$ as in
Lemma~\ref{Lemma-NonErgInvMeasure}.
It is clear that  the family  $(W^n)_{n\geq 0}$ 
of finite dimensional linear subspaces is  decreasing. So, there exists
$N\geq 0$ such that $W^{n}= W^{N}$ for all $n\geq N.$
Set $W:= W^{N}$. Recall 
that $V_0\subset W^n$ for every $n\geq 0$ and hence 
$V_0\subset W.$

\vskip.2cm
$\bullet$ {\it Second step.} We claim that  $\mu(\vfi(W(\AA))\geq c.$

Indeed, since $\mathcal{N}(V_0,S_n)\subset W^n(\QQ_{S_n}),$ it follows from the first step
that
$$\mu_n(\vfi_{S_n}(W(\QQ_{S_n})))=\mu_n(\vfi_{S_n}(W^n(\QQ_{S_n}))) \geq c$$
for every $n\geq N.$
Setting 
$$\Omega_n:=\vfi\left(W(\QQ_\infty) \times  \prod_{p\in S_n, p\neq \infty} 
W(\ZZ_p) \times \prod_{p\notin S_n} \ZZ_p^d\right),
$$
this means that 
$$\mu(\Omega_n)\geq c \tout n\geq N,$$
 since  $\mu_n$ is the image of $\mu$ under $\pi_{S_n}.$

As  $(\Omega_n)_{n\geq N} $ is a decreasing sequence, it follows that 
$$
\mu\left(\bigcap_{n\geq N} \Omega_{n}\right)\geq c.
$$
On the other hand, we have (see Lemma~\ref{lem-Intersection})
 $$\bigcap_{n\geq N} \Omega_{n}=\vfi(W(\AA))$$
 and the claim is proved.
 \vskip.2cm
 
Set $Y:=\vfi(V_0(\AA)).$ 
\vskip.2cm
$\bullet$ {\it Third step.} We claim that there exists 
$x\in \vfi(W(\AA))$ such that $\mu( \vfi(x+Y))=1.$

Indeed, $\vfi(W(\AA))$ is $\GG(\QQ)$-invariant, since $W$ is $\GG(\QQ)$-invariant.
By the ergodicity of $\mu,$  it follows that $\mu(\vfi(W(\AA))=1$; so, we
may view $\mu$ as probability measure on $\vfi(W(\AA)).$ 

Let $Z\subset \vfi(W(\AA))$ be the support of $\mu.$ Again by ergodicity of $\mu,$
there exists a point $a\in W(\AA)$ such that $x=\vfi(a)\in Z$ 
and such that the  $\GG(\QQ)$-orbit of $x$ is dense in $Z.$ Since $g(a)\in a+V_0(\AA)$ for 
all $g\in \GG(\QQ),$ this means that $x+Y=Z$ and so $\mu(x+Y)=1.$

\vskip.2cm
$\bullet$ {\it Fourth step.} We claim that $\mu$ is invariant under  translations by elements from $Y.$ Once proved, it will follow that $\mu=\mu_{x+Y},$
by the uniqueness of the Haar measure on the closed subgroup $Y$ of $X.$

Indeed, the topological space $\vfi(W(\AA))\subset X$ is the projective limit of the sequence $(\vfi(W(\QQ_{S_n})))_{n\geq 0}$ of the topological spaces $\vfi(W(\QQ_{S_n}))\subset X_{S_n},$ with respect to the canonical 
maps $\vfi(W(\QQ_{S_n})) \to \vfi(W(\QQ_{S_m}))$
for $n\geq m.$ So,  the sets of the form $\vfi (B\times \prod_{p\notin S_n} W(\ZZ_{p}))$, where $B$ runs over the Borel subsets of $W(\QQ_{S_n})$, generate the Borel structure of $\vfi(W(\AA)).$

For every $n\geq 0$, the image $\mu_n$ of $\mu$ in $X_{S_n}$ is the 
Haar measure on the coset $\pi_{S_n}(x+Y)$ of the subgroup $\vfi_{S_n}(V(\QQ_{S_n})).$
So, 
$\mu_n$ is invariant under  translations by elements from $\vfi_{S_n}(V(\QQ_{S_n})).$
This means that, for every Borel subset $B$ of $W(\QQ_{S_n})$, we have
$$
\mu\left(z+\vfi \left(B\times \prod_{p\notin S_n} W(\ZZ_{p})\right)\right)=
\mu\left(\vfi \left(B\times \prod_{p\notin S_n} W(\ZZ_{p})\right)\right),
$$
for every $z\in Y.$ This proves the claim.
\end{proof}

\subsubsection{Orbit closures}

We now deduce  the description of orbit closures of points in $X$
from the corresponding description in the $S$-adic case.

Recall that $X_S= \QQ_S^d/\ZZ[1/S]^d$ for $S\in  \mathcal{S}$ and that 
$\vfi: \AA^d\to X$, $\vfi_S:\QQ_S^d\to X_S$ , and
$\pi_S: X\to X_S$ denote the canonical projections.
\begin{theorem}
 \label{Theo-OrbitClosure-Adelic}
 Let $\GG$ be a connected algebraic subgroup of $GL_d$  defined over $\QQ.$ Assume that $\GG(\QQ)$ is generated by unipotent  one-parameter  subgroups. Let $a\in \AA^d$ and $x=\vfi(a)\in X.$
 There exists a  $\GG(\QQ)$-invariant linear subspace $V_0$ of $\QQ^d$ 
 such that the closure of the $\GG(\QQ)$-orbit of $x$   in $X$ coincides with $x+\vfi(V_0(\AA)).$
 \end{theorem}
 \begin{proof} 
 For $S\in \mathcal{S},$  let $x_S= \pi_S(x)\in X_S.$
 By  Proposition~\ref{Prop-OrbitClosure-S-adic},
 there  exists a unique  $\GG(\QQ)$-invariant linear subspace $V_S$ of $\QQ^d$ 
 such that  
 $$\overline{\GG(\ZZ[1/S])x_S}=x_S+\vfi_S(V_S(\QQ_S)).$$

  Fix $S_0\in \mathcal{S}$  such   that $V_0:=V_{S_0}$ has maximal
  dimension among all the subspaces $V_S$ for  $S\in \mathcal{S}.$
  
Let $\P\cup\{\infty\}=\cup_{n\geq 0} S_n$  for an increasing  sequence of
subsets $S_n \in \mathcal{S}$. Let $n\geq 1$
and write $V_n$ for $V_{S_n}.$

Since   $\{g(x_{S_0})-x_{S_0}\mid g\in \GG(\ZZ[1/S_0])\}$ is dense in $\vfi_S(V_0(\QQ_{S_0}))$ and since  
$$g(x_{S_n})-x_{S_n} \in \vfi_{S_n}(V_n(\QQ_{S_n})) \tout g\in \GG(\ZZ[1/S_0]),$$
it follows that 
$V_0 \subset V_n$ (see the first step in the proof of Theorem~\ref{Theo-ErgodicityAdelic})
and hence $V_n=V_0,$ by maximality of the dimension
of $V.$ Therefore, we have 
$$\overline{\GG(\ZZ[1/S_n])x_{S_n}}=x_{S_n}+\vfi_{S_n}(V_0(\QQ_{S_n})) \tout n\geq 0.$$
  
  We claim that  the  closure of the $\GG(\QQ)$-orbit of $x$   in $X$ coincides with $x+\vfi(V_0(\AA)).$

 Indeed, let $n\geq 0$ and   $g\in \GG(\ZZ[1/S_n])$.
 For every  $m\geq n,$ we have
 $g(x_{S_m})-x_{S_m} \in \vfi_{S_m}(V_0(\QQ_{S_m}))$
 and hence
 $$
 g(x)-x\in \vfi\left(V_0(\RRR) \times  \prod_{p\in S_m, p\neq \infty} V_0(\ZZ_p) \times \prod_{p\notin S_m} \ZZ_p^d\right)
  $$
 It follows  (see Lemma~\ref{lem-Intersection}) that $g(x)-x\in \vfi(V_0(\AA))$ for every 
 $g\in \GG(\ZZ[1/S_n])$. Therefore, $x+\vfi(V_0(\AA))$ is $\GG(\QQ)$-invariant.
  Since $x+\vfi(V_0(\AA))$ is closed in $X,$ this implies that
  $$\overline{\GG(\QQ)x} \subset x+\vfi(V_0(\AA)).$$
  
  Conversely, let $y \in \vfi(V_0(\AA)).$ 
  Then $y=\vfi(v)$ for $v=(v_p)_{p\in \P\cup \{\infty\}}$ in $V_0(\AA)$
  with  $v_p\in V_0(\ZZ_p)$ for all $p\in \P.$ 
     Let $U$ be a neighbourhood of 
  $y$ in $X.$ Then $U$ contains a set of the form
  $ \vfi\left(O_n\times \prod_{p\notin S_n} \ZZ_p^d \right)$
  for some $n\geq 0,$ where  $O_n$ is a neighbourhood
  of $(v_p)_{p\in S_n} $  in $\RRR^d\times \prod_{p\in S_n\setminus \{\infty\}}\ZZ_p^d$
  
  Since $\GG(\ZZ[1/S_n])x_{S_n}-x_{S_n}$  is dense in $\vfi_{S_n}(V_0(\QQ_{S_n}))$, there exists
  $g\in \GG(\ZZ[1/S_n])$ such that $g(x_{S_n})-x_{S_n}\in \vfi_{S_n}(O_n).$
  As $g(\ZZ_p^d)\subset \ZZ_p^d$ for every $p\notin S_n,$ it follows that  
  $$
  g(x)- x\in \vfi\left(O_n\times \prod_{p\notin S_n} \ZZ_p^d \right) \subset U.
  $$
  This shows that $x+y\in \overline{\GG(\QQ)x}.$

 \end{proof}

\section{Proof of Theorem~\ref{Theo-GenAlgGroup}}
\label{S:GenAlgGroup}
In this section, we will give the proof of Theorem~\ref{Theo-GenAlgGroup}.
\subsection{Invariant characters on $\QQ^d$}
\label{SS:InvCharQd}
Let $\GG$ be a connected algebraic subgroup of $GL_d$  defined over $\QQ.$
Using Fourier transform, we  establish the dual versions of Theorems~\ref{Theo-ErgodicityAdelic} and ~\ref{Theo-OrbitClosure-Adelic} in terms of $\GG(\QQ)$-invariant characters on  $\QQ^d$.

Recall (see Subsection~\ref{SS:RestrictionRadical}) that, after the choice of  nontrivial unitary character $e$ of $\AA$ which is trivial on  $\QQ,$ we can identify  $\widehat{\QQ^d}$ with $X=\AA^d/\QQ^d$ by means of the $GL_d(\QQ)$-equivariant map  
$$X\to \widehat{\QQ^d},\quad  a+\QQ^d\mapsto \la_a,$$ 
where 
$$
\lambda_{a}(q)= e(\langle a, q\rangle) \quad \tout q=(q_1, \dots, q_d) \in \QQ^d,
$$
and $\langle a, q\rangle= \sum_{i=1}^d a_i q_i$ for $a=(a_1, \dots, a_d)\in \AA^d.$

By Pontrjagin duality, the map 
$$q\mapsto (a+\QQ^d\mapsto \lambda_{a}(q))$$
 is a $GL_d(\QQ)$-equivariant isomorphism between $\QQ^d$ and the dual group 
 $\widehat{X}$ of $X.$ 
The annihilator of a subset $Y$ of $X$ in $\QQ^d\cong \widehat{X}$ is 
$$
Y^\perp=\left\{q\in \QQ^d\mid \la_a(q)=1 \quad \text{for all} \quad a+\QQ^d\in Y\right\}.
$$
If $Y$ is a closed subgroup of $X,$  the map 
$$ q+ Y^\perp\mapsto (a+\QQ^d\mapsto \lambda_{a}(q))$$
is an isomorphism between  $\QQ^d/ Y^{\perp}$ and $\widehat{Y}.$

We will need the following characterization of subsolenoids of $X.$
Recall that $\vfi$ is the projection $\AA^d\to X.$
\begin{lemma}
\label{Lem-Subsolenoids}
Let $\mathcal{Y}(X)$ be the set of connected and closed subgroups of $X$ and $\mathbf{Gr}(\QQ^d)$ the set of linear subspaces of $\QQ^d.$
\begin{itemize}
\item[(i)]  The map $Y\to Y^\perp$ is a $GL_d(\QQ)$-equivariant bijection between $\mathcal{Y}(X)$ and $\mathbf{Gr}(\QQ^d)$.
\item[(ii)] For $V\in \mathbf{Gr}(\QQ^d),$ we have  $\vfi(V(\AA))\in \mathcal{Y}(X)$; moreover, we have
$$\vfi(V(\AA))\cap \vfi(W(\AA))=\vfi((V\cap W)(\AA))$$
 for every $V, W\in \mathbf{Gr}(\QQ^d).$
\item[(iii)] The map $V\to \vfi(V(\AA))$ is a $GL_d(\QQ)$-equivariant bijection between $\mathbf{Gr}(\QQ^d)$ and $ \mathcal{Y}(X).$
\end{itemize}
In particular, for every $P\in \mathbf{Gr}(\QQ^d),$ there exists a unique $V\in \mathbf{Gr}(\QQ^d)$ such that $P= \vfi(V(\AA))^\perp.$
\end{lemma} 
\begin{proof}
(i) Let $Y$ be a closed subgroup of $X.$ Then $Y$ is connected
if and only $\widehat{Y}$ is torsion-free (see  Corollary (24.19) in\cite{HewittRoss1}),
 that is, if and only if $\QQ^d/Y^\perp$ is torsion-free. It follows that $Y$ is connected
if and only if $Y^\perp$ is a linear subspace of $\QQ^d.$ 

\noindent
(ii) Let $V\in \mathbf{Gr}(\QQ^d)$. Since the canonical embedding of $\RRR$ is dense in $\AA/\QQ$, the subgroup $\vfi(V(\RRR))$ is dense in $\vfi(V(\AA))$ and therefore $\vfi(V(\AA))$ is connected. Moreover,  $\vfi(V(\AA))$ is closed as it is the continuous
image of  the compact solenoid $V(\AA)/V(\QQ)$.

Let  $V, W\in \mathbf{Gr}(\QQ^d)$. It is clear that $\vfi((V\cap W)(\AA))$
is contained in $\vfi(V(\AA))\cap \vfi(W(\AA)).$
Let $a \in V(\AA)$ be such  that $\vfi(a)\in \vfi(W(\AA)).$ 
Writing $a= (a_p)_{p\in \P\cup \{\infty\}},$
we may assume that $a_p\in V(\ZZ_p)$ for every $p\in \P$.
So, there exists $q\in \QQ^d$ such that $a_{\infty}+q\in W(\RRR)$ and
$a_p+ q\in W(\ZZ_p)$ for every $p\in \P.$ It follows
that $q\in \ZZ^d$. Hence, we have
$a_{\infty}\in W(\RRR)+\ZZ^d.$ 
Observe that  $ta\in W(\AA)$ for every $t\in \QQ;$ it follows
that $ta_{\infty}\in W(\RRR)+\ZZ^d$ for every $t\in \QQ$ and hence 
for every $t\in \RRR,$ since $W(\RRR)+\ZZ^d$ is closed in $\RRR^d.$
By connectedness, this implies that $a_{\infty}\in W(\RRR).$ 
Since $a_{\infty}+q\in W(\RRR),$ we have $q\in W(\ZZ)$ and so
$a_p\in W(\ZZ_p)$ for every $p\in \P.$ This shows that $a\in (V\cap W)(\AA)$.
So, $\vfi(V(\AA))\cap \vfi(W(\AA))=\vfi((V\cap W)(\AA)).$

\noindent
(iii)
 As already mentioned (see the proof of Lemma~\ref{Lemma-GenericSingular}), 
 we have $\vfi(V(\AA))\neq \vfi(W(\AA))$ for   every $V,W\in \mathbf{Gr}(\QQ^d)$
with $V\neq W.$

Let $Y\in \mathcal{Y}(X)$. Then $Y^\perp\in\mathbf{Gr}(\QQ^d),$ by (i).
Consider the linear subspace 
$$
V:=\{a\in \QQ^d\mid  \langle a, q\rangle=0 \quad \text{for all} \quad  q\in Y^\perp\}.
$$
of $\QQ^d.$  We claim that $Y= \vfi(V(\AA)).$

Indeed, let $q_1, \dots, q_s$ be a basis of $Y^\perp$. For every $i\in\{1,\dots, s\}$ and $t\in \QQ,$
let $V_i=\{a\in \QQ^d\mid  \langle a, q_i\rangle=0 \}$ and choose $a_{i,t}\in \QQ^d$  such that
$\langle a_{i,t}, q_i\rangle=t$.
Then
$$
\{a\in \AA^d\mid \langle a, q_i\rangle=t \}= V_i(\AA)+a_{i,t}.
$$
We have
$$
\begin{aligned}
Y&= (Y^\perp)^\perp=\{a+\QQ^d\in X\mid \la_a(q)=1 \quad \text{for all} \quad  q\in Y^\perp\}\\
&= \{a+\QQ^d\in X \mid \la_{a}(tq)=1 \quad \text{for all} \quad  q\in Y^\perp, t \in \QQ\}\\
&= \{a+\QQ^d\in X\mid e(t\langle a, q\rangle)=1 \quad \text{for all} \quad  q\in Y^\perp, t\in \QQ\}\\
&=\{a+\QQ^d\in X \mid \langle a, q_i\rangle \in \QQ \quad \text{for every} \quad  i=1,\cdots, s\}\\
&=\bigcap_{i=1}^s \bigcup_{t\in \QQ} \vfi (V_i(\AA)+a_{i,t})=\bigcap_{i=1}^s \vfi (V_i(\AA)) 
\end{aligned}
$$
Using  (ii), it follows that $Y=\vfi\left((\bigcap_{i=1}^s V_i)(\AA)\right)= \vfi(V(\AA)).$

\end{proof}

Let $\GG$ be a connected algebraic subgroup of $GL_d$  defined over $\QQ.$
For $a\in \AA^d,$ let $P_a$ be the $\QQ$-linear span
of $\{\la_{g(a)}\la_{-a}\mid g\in \GG(\QQ)\}^\perp$
in  $\QQ^d;$ so,
$$
\begin{aligned}
P_a&= 
\bigcap_{g\in \GG(\QQ), t\in \QQ}\{q\in \QQ^d\mid \la_{g(a)} (tq)= \la_{a}(tq)\} \\
&=\bigcap_{g\in \GG(\QQ)}\{q\in \QQ^d\mid \langle g(a)-a, q\rangle\in \QQ\}\\
\end{aligned}
$$
and $P_a$ is a  $\GG(\QQ)$-invariant linear  subspace of $\QQ^d$. Define 
$\chi_a: \QQ^d\to \CC$ by  
$$\chi_a (q)=\begin{cases}
\la_a(q)&\text{if } q\in P_a \\
0&\text{otherwise}.
\end{cases}
$$
Observe that 
$\chi_a$ is $\GG(\QQ)$-invariant  and is of positive type
(see Proposition~\ref{Pro-InducedTrace}), that is, 
$\chi_a\in \Tr(\QQ^d, \GG(\QQ)).$

Recall that two points $x,y\in X$ belongs to the same $\GG(\QQ)$-quasi-orbit 
if their  $\GG(\QQ)$-orbits  have the same closure in $X.$
\begin{theorem}
\label{Theo-ErgodicityAdelic-Dual}
Let $\GG$ be a connected algebraic subgroup of $GL_d$  defined over $\QQ.$
Assume that $\GG(\QQ)$ is generated by unipotent  one-parameter 
 subgroups.
The  map  
$$\AA^d\to \Tr(\QQ^d, \GG(\QQ)), \quad a\mapsto \chi_a$$
has $\Char(\QQ^d, \GG(\QQ))$ as image and factorizes to a bijection
$$
X/{\sim} \to \Char(\QQ^d, \GG(\QQ)),
$$
where $X/{\sim}$ is the space of $\GG(\QQ)$-quasi-orbits in $X=\AA^d/\QQ^d.$
\end{theorem}
\begin{proof}
Identifying  $\QQ^d$ with
 $\widehat{X}$, the  Fourier transform on $X$ 
 is the map
$$\mathcal{F}: \Prob(X) \to \Tr(\QQ^d)$$
given by 
$$
\mathcal{F}(\mu)(q)= \int_{X} \la_a(q) d\mu(a+\QQ^n) \tout \mu \in \Prob(X), q\in \QQ^d.
$$
Recall (see Proposition~\ref{Prop-Abelian})  that $\mathcal{F}$ restricts to a bijection 
$$\mathcal{F}: \Prob(X)^{\GG(\QQ)}_{\rm erg} \to \Char(\QQ^d, \GG(\QQ)).$$

Let $a\in \AA^d.$ By Lemma~\ref{Lem-Subsolenoids},
there exists a   $\GG(\QQ)$-invariant linear subspace $V$ of $\QQ^d$
such that $P_a=Y^\perp$ for $Y=\vfi(V(\AA))$. 

Let $\mu=\mu_{x+Y}$ for $x=\vfi(a)$.

\vskip.2cm
$\bullet$ {\it First step.} We claim that $\mathcal{F}(\mu_{x+Y})= \chi_a.$

Indeed, for every $q\in \QQ^d,$ we have 
$$
\begin{aligned}
\mathcal{F}(\mu_{x+Y})(q)&=\int_{Y} \la_{a+b}(q) d\mu_Y(b+\QQ^d)\\
&=\la_a(q) \int_{Y} \la_{b}(q) d\mu_Y(b+\QQ^d).
\end{aligned}
$$
Now, $\int_{Y} \la_{b}(q) d\mu_Y(b+\QQ^d)=0$, whenever
$b+\QQ^d \mapsto \la_b(q)$ is a non-trivial character of $Y$, by the
orthogonality relations.
This proves the claim.

\vskip.2cm
$\bullet$ {\it Second step.}  We claim that $\chi_a \in \Char(\QQ^d, \GG(\QQ))$.
In view of step one, it suffices to show that  $\mu_{x+Y}$ is ergodic

Assume, by contradiction, that $\mu_{x+Y}$ is not ergodic.
Observe that $\GG(\QQ)$ is connected (in the Zariski topology of $GL_d(\QQ)$) and has therefore no proper  finite index subgroup.
Therefore, by Proposition~\ref{Prop-ErgodicityAffineSubspace},
there exists a proper closed  connected subgroup $Z$ of $Y$
 such that $g(x+y)\in x+y+ Z$ for every $g\in \GG(\QQ)$ and  $y\in Y.$
 By Lemma~\ref{Lem-Subsolenoids}, we can write
 $Z=\vfi(W(\AA))$ for a  $\GG(\QQ)$-invariant proper $\QQ$-linear subspace of $V$.
 So, we have $g(x)-x\in \vfi( W(\AA))$ for every $g\in \GG(\QQ)$.
 This implies that
 $$\vfi(W(\AA))^\perp=P_{a}= Y^\perp=\vfi(V(\AA))^\perp,$$
 which is a contradiction, since $W\neq V.$
 
 \vskip.2cm
$\bullet$ {\it Third step.}  We claim that the closure of the $\GG(\QQ)$-orbit of 
$x$ coincides with $x+Y.$

Indeed,  by Theorem~\ref{Theo-OrbitClosure-Adelic},
 there exists a  $\GG(\QQ)$-invariant linear subspace $W$ of $\QQ^d$ 
 such that the closure of the $\GG(\QQ)$-orbit of $x$   in $X$ coincides with 
 $x+\vfi(W(\AA)),$ that is, $g(x)-x\in \vfi( W(\AA))$ for every $g\in \GG(\QQ)$.
 As in the second step, this implies that $W=V.$

  \vskip.2cm
$\bullet$ {\it Fourth step.}  We claim that $ \Char(\QQ^d, \GG(\QQ))= \{\chi_a\mid a\in \AA^d\}.$ Indeed, this follows from Theorem~\ref{Theo-ErgodicityAdelic} and the first two 
steps.

\vskip.2cm
$\bullet$ {\it Fifth step.} Let $a_1, a_2\in \AA^d$. We claim that $\chi_{a_1}=\chi_{a_2}$ 
if and only if $x_1=\vfi(a_1)$ and $x_2=\vfi(a_2)$ belong to the same $\GG(\QQ)$-quasi-orbit.

Indeed, observe that $P_{g(a)}=P_a$ and $\chi_{g(a)}= \chi_{a}$ for every 
$a\in \AA^d$ and $g\in \GG(\QQ)$
It follows that $\chi_{a_1}=\chi_{a_2}$   if $x_1$ and $x_2$ belong to the same quasi-orbit.

Conversely, assume that $\chi_{a_1}=\chi_{a_2}$.
Then $P_{a_1}= P_{a_2}$ and $\la_{a_1}=\la_{a_2}$ on $P_{a_1}.$
So, $Y_1=Y_2$, where $Y_i= P_{a_i}^\perp$ for $i=1,2$, and $x_1-x_2\in Y_1.$
Hence, $x_1+Y_1=x_2+Y_2$ and so, by the third step,  $x_1$ and $x_2$ belong to the same quasi-orbit.

\end{proof}
 \subsection{Conclusion of the proof of Theorem~\ref{Theo-GenAlgGroup}}
  \label{SS:FinPreuve}
  Let $G=\GG(\QQ)$ be as in the statement of Theorem~\ref{Theo-GenAlgGroup},
$G=LU$  a Levi decomposition of $G$, and $\mathfrak{u}$ the Lie algebra of $U.$

Let $\psi\in \Char(G)$.
 \vskip.2cm
$\bullet$ {\it First step.} 
Set 
$\vfi:= \psi|_{U}\circ \exp.$ 
There exists $\la\in \widehat{u}$ such that $\vfi$ coincides with  the trivial extension to $\mathfrak{u}$ of the restriction of $\la$ to $\mathfrak{p}_\la,$
where $\mathfrak{p}_\la$ is the $G$-invariant linear subspace of $\mathfrak{u}$ 
given by 
$$\mathfrak{p}_\la=\{X\in \mathfrak{u}\mid \lambda(\Ad(g)(tX))=\lambda(tX) \tout g\in G, t\in \QQ\}.$$

Indeed, as discussed in Subsection~\ref{SS:RestrictionRadical},
$\vfi\in \Char(\mathfrak{u}, G)$. The claim follows therefore from Theorem~\ref{Theo-ErgodicityAdelic-Dual}.

\vskip.2cm
$\bullet$ {\it Second step.} We claim that 
$$\psi(\exp (X) g)= \psi(\exp(X))\psi(g) \tout g\in G,  X\in \mathfrak{p}_\lambda.$$

Indeed, since 
$$|\psi(\exp (X))|=|\vfi(X)|= |\la(X)|=1$$
for every $ X\in \mathfrak{p}_\lambda,$  the claim follows from 
Proposition~\ref{Proposition-KernelProjectiveKernel}.ii.

\vskip.2cm

$\bullet$ {\it Third  step.} For every  $g\in G$ and $X\in \mathfrak{u},$
we have 
$$\psi(g)=\psi\left(\exp(-X) \exp(\Ad(g)(X)) g\right).$$

This is indeed the case, since
$$
\begin{aligned}
\psi(g)&=\psi(\exp(-X)g\exp(X))\\
&=\psi(\exp(-X) \exp(\Ad(g)(X)) g).
\end{aligned}.
$$

\vskip.1cm
Recall that 
$$
\mathfrak{k}_\lambda=\left\{X\in \mathfrak{u}\mid 
\la(\Ad(g)(tX))=1 \tout g\in G, t\in \QQ\right\};
$$
observe that $K_\la:=\exp(\mathfrak{k}_\la)$ is in general strictly
contained in  $K_\psi\cap U,$ where 
$$K_\psi:=\{g\in G\mid \psi(g)=1\}.$$

Let 
$$G_\la=\left\{g\in G\mid \Ad(g)(X)\in X+\mathfrak{k}_\lambda \tout X\in  \mathfrak{u}\right\}.$$

\vskip.2cm
Assume that $G_\la\neq G$. Observe that this implies that $\la\neq 1_{\mathfrak{u}}.$
Let $g\in G\setminus G_\la$ and 
fix  $X\in  \mathfrak{u}$ such that $\Ad(g)(X)-X
\notin \mathfrak{k}_\lambda.$ Let 
$$
A_{g,X}:=\left\{t\in \QQ\mid \exp(-tX) \exp(\Ad(g)(tX)\in K_\psi\right\}.
$$
By Lemma~\ref{Lem-CommutarorAndSubgroup}, $A_{g,X}$ is a subgroup of $\QQ.$

$\bullet$ {\it Fourth step.} 
 We claim that $A_{g,X}\neq \QQ.$

Indeed, assume, by contradiction, that $A_{g,X}=\QQ,$ that is, 
$$ \exp(-tX) \exp(\Ad(g)(tX)\in K_\psi \tout t\in \QQ.$$ 
By the Campbell-Hausdorff  formula,   there exists $Y_1, Y_2, \dots, Y_r\in \mathfrak{u}$ 
such that 
$$
 \exp(-tX) \exp(\Ad(g)(tX))= \exp\left(\sum_{k=1}^r t^k Y_k\right) \tout t\in \QQ,
 $$
 where $Y_1= \Ad(g)(X)-X.$  We have then
 $$
 \la(\sum_{k=1}^r t^k Y_k)=1 \tout t\in \QQ.
 $$
 
Identifying $\mathfrak{u}$ with $\QQ^d$ via a basis $\{X_1, \dots X_d\},$ the
character $\la$ of  $\mathfrak{u}$ is given by some  $a=(a_1, \dots, a_d)\in \AA^d$ via the formula
 $$
 \la (\sum_{i=1}^d q_i X_i)=e( \sum_{i=1}^d a_i q_i)  \tout (q_1, \dots, q_d) \in \QQ^d
 $$
 for a nontrivial unitary character $e$ of $\AA$ which is trivial on  $\QQ$
 (see Subsection~\ref{SS:RestrictionRadical}).
 It follows that 
 $$
 e\left(\sum_{k=1}^r t^k(\sum_{i=1}^d a_i q_{k,i})\right) =1\tout t\in \QQ,
 $$
 where $(q_{k,i})_{i=1}^d$ are the coordinates of $Y_k$ in $\{X_1, \dots X_n\}.$
 This implies that $\sum_{i=1}^d a_i q_{k,i} \in \QQ$ for every $k=1, \dots, r.$
 Indeed, otherwise the image of the set
 $$\{\sum_{k=1}^r t^k(\sum_{i=1}^d a_i q_{k,i})\mid t\in \QQ\}$$
 would be dense in $\AA/\QQ$ (see \cite[Theorem 2]{Corwin-Pfeffer}
 or \cite[Theorem 5.2]{Bergelson}) and this would contradicts the non-triviality of $e.$
 
 In particular, we have $\sum_{i=1}^d a_i q_{1,i} \in \QQ$ and, since $Y_1= \Ad(g)X-X,$ 
 we obtain $\la (\Ad(g)(tX)-tX)=1$ for all $t\in \QQ.$ Therefore,  $\Ad(g)X-X
\in \mathfrak{k}_\lambda$ and this is a contradiction.
 
\vskip.2cm
$\bullet$ {\it Fifth  step.} 
Let  $g\in G\setminus G_\la$. 
We claim that $\psi(g)=0.$

Indeed, let $X\in \mathfrak{u}$ with $\Ad(g)(X)-X
\notin \mathfrak{k}_\lambda$ and 
let  $A_{g,X}\subset \QQ$ be as in the fourth step. Set
$$
B_{g, X}:=\left\{t\in \QQ\mid \exp(-tX) \exp(\Ad(g)(tX)\in P_\psi\right\}.
$$
Then $A_{g, X}\subset B_{g, X}$ and, by Lemma~\ref{Lem-CommutarorAndSubgroup} again, $B_{g, X}$ is a subgroup of $\QQ.$
Two cases may occur.

{--First case:} $A_{g, X}\neq B_{g, X}.$ So, there exists $t\in \QQ$ such that 
$$\exp(-tX) \exp(\Ad(g)(tX)\in P_\psi\setminus K_\psi.$$
Then, using the third and second steps, we have
$$
\begin{aligned}
\psi(g)=\psi(\exp(-X) \exp(\Ad(g)(X)) \psi(g)
\end{aligned}
$$
and hence $\psi(g)=0,$ since $\psi( (\exp(-X) \exp(\Ad(g)(X))\neq 1.$

  {--Second case:} $A_{g, X}=B_{g, X}.$ Then $ B_{g, X}$ is a proper subgroup of 
$\QQ,$ by the the fourth step. So, 
$B_{g, X}$ has infinite index in $\QQ$ an we can find an infinite sequence
$(t_n)_{n\geq 1}$ in $\QQ$ such that 
$t_n-t_m\notin B_{g, X}$ for all $n\neq m.$

Set 
$$u_n:=\exp(-t_nX)\exp(\Ad(g)(t_nX)) \tout n\geq 1.$$
 For $n\neq m,$ we have
 $$
 \begin{aligned}
 u_n u_m^{-1} 
 &=\exp(-t_nX)\exp(g(t_n-t_m)X))\exp(t_mX)\\
 &=\exp(-t_mX)\left(\exp(-(t_n-t_m)X)\exp(\Ad(g)(t_n-t_m)X)\right)\exp(t_mX);
\end{aligned}
$$
 since $\exp(-(t_n-t_m)X)\exp(\Ad(g)(t_n-t_m)X\notin P_\psi$ and since 
 $P_\psi$ is a normal subgroup of $U,$ we have therefore
 $u_n u_m^{-1} \notin P_\psi$ and hence
 $$
 \psi(u_n u_m^{-1})=0 \tout n\neq m,
 $$
 by the first step.
 
 As $u_n$ coincides with the commutator $[\exp(t_n X), g]$ in $G,$
it follows from Lemma~\ref{Lem-Hilbert} that $\psi(g)=0.$

\vskip.2cm
It remains to determine the restriction of $\psi$ to $G_\la.$

Since $\psi|_{K_\la}=1_{K_\la},$ we may view $\psi$ as character
of 
$G/K_\lambda,$ which is the group of $\QQ$-points of an algebraic group,
and we can therefore assume that  $K_\la=\{e\}.$
Then $G_\la$ is the centralizer of $U$ in $G$
and is the group of $\QQ$-points of an algebraic normal subgroup of $\mathbf{G}$.
Let $G_\la=L_1 U_1$ be a Levi decomposition of $G_\la.$ Since $U_1$ is a unipotent characteristic subgroup of $G_\la,$  we have $U_1\subset U$.
Moreover, $L_1$ is  the group of $\QQ$-points of an algebraic subgroup $\mathbf{L}_1$
of $\mathbf{G}$ and so $\mathbf{L}$ is contained in a Levi subgroup of $\mathbf{G}.$ As
two Levi subgroups of $\mathbf{G}$ are conjugate by an element of $U$ (see \cite{Mostow}), we 
can assume that $L_1\subset L,$ that is, $L_1=L_\la$ and so, 
$G_\la=  L_\la Z(U).$ 

\vskip.2cm
$\bullet$ {\it Sixth step.}  We claim that 
there exists $\vfi_1\in \Char(L_\la)$ such that 
$$
\psi(gu)= \vfi_1(g) \psi(u) \tout g\in L_\la, u\in U.
$$

Indeed,
 we can find a normal subgroup $H$ of $L$ which centralizes 
$L_\la$ such that $L_\la \cap H$ is finite and such that
$L=L_\la H$ (see Proposition~\ref{Prop-CharProduct2} below).
Then $G= L_\la H U$ and $H U$ centralizes $L_\la.$ So,
the claim follows from Proposition~\ref{Prop-CharProductGroup}
(see also Corollary~\ref{Cor-CharProduct}).

 \vskip.2cm
$\bullet$ {\it Seventh step.} Let $ \la\in \widehat{\mathfrak{u}}, \vfi\in \Char(L_\la),$
and let $\Phi_{(\la, \vfi)}: G\to \CC$ be defined as in Theorem~\ref{Theo-GenAlgGroup}.
We claim that $\psi:=\Phi_{(\la, \vfi)}\in \Char(G).$

Indeed, it is clear (see Proposition~\ref{Pro-InducedTrace}) that $\psi\in \Tr(G)$.
Write 
$$\psi= \int_{\Omega} \psi_\omega d\nu(\omega)$$ as an integral 
over a probability space $(\Omega, \nu)$ with $\psi_\omega\in \Char(G)$
for every $\omega$ (see Remark~\ref{Rem-Choquet}). 

Then  $\vfi:=\psi|_U \circ \exp$ coincides with  the trivial extension to $\mathfrak{u}$ of the restriction of $\la$ to $\mathfrak{p}_\la,$
where $\mathfrak{p}_\la$ is defined as above.
It follows from  Theorem~\ref{Theo-ErgodicityAdelic-Dual}
that  $\vfi\in \Char(\mathfrak{u}, G)$. This implies that the restriction of 
$\psi_{\omega}$ to $U$ coincides with $\psi|_U$ for ($\nu$-almost)  every 
$\omega.$ 

Let $\omega\in \Omega.$  The fifth step, applied to $\psi_{\omega} \in \Char(G),$ shows that
$\psi_\omega=0$ on $G\setminus G_\la,$ where $G_\la$
is defined as above. By the sixth step, also applied to $\psi_{\omega},$
there exists 
$\vfi_1^{\omega} \in \Char(L_\la)$ such that 
$$\psi(gu)= \vfi_1^\omega(g) \psi(u) \tout g\in L_\la, u\in U.
$$

As a result, we have 
$$
\vfi_1= \int_{\Omega} \vfi_1^\omega d\nu(\omega).
$$
Since $\vfi_1\in \Char(L_\la),$ it follows that $\vfi_1^{\omega}=\vfi_1$ 
and hence that $\vfi=\vfi_\omega$ 
for ($\nu$-almost)  every  $\omega.$ This shows that $\vfi\in \Char(G).$

$\bullet$ {\it Eighth step.} 
 Let $\la_1, \la_2\in \widehat{\mathfrak{u}}$ and $ \vfi_1\in \Char(L_{\la_1}),
 \vfi_2\in \Char(L_{\la_2}).$ 
 We claim that $\Phi_{(\la_1, \vfi_1)}= \Phi_{(\la_2, \vfi_2)}$
 if and only if $\la_1$ and $\la_2$ have the same $G$-orbit closure and $\vfi_1=\vfi_2.$
 
 Indeed, set $\psi_i= \Phi(\la_i, \vfi_i)$ for $i=1,2.$ It follows from Theorem~\ref{Theo-ErgodicityAdelic-Dual}
 that $\psi_1|_U= \psi_2|_U $ if and only if the closures of the $G$-orbits of  $\la_1$ and $\la_2$  coincide. If this is the case, then $\mathfrak{k}_{\la_1}=\mathfrak{k}_{\la_2}$
 and hence $G_{\la_1}=G_{\la_2}.$ 
 The claim follows from this facts.

 \subsection{Characters of semi-simple algebraic groups}
 \label{SS: ComputationSS}
 The following proposition, in combination with \cite{Bekka} and    Corollary~\ref{Cor-CharProduct}, shows how the  characters of  the groups $L_\la$ appearing in Theorem~\ref{Theo-GenAlgGroup} can be described.
   
A group $L$ is the  \textbf{almost direct product} of subgroups  $H_1, \dots, H_n$ of $L$  if the product map $H_1\times\cdots\times H_n\to  L$ 
 is a surjective homomorphism with  finite kernel.
 
  \begin{proposition}
  \label{Prop-CharProduct2}
  Let $\GG$ be a connected semi-simple algebraic group defined over a field $k$.
 Assume that  $\GG(k)$ is generated by its unipotent one-parameter subgroups.
  Let $\LL$ be a (non necessarily connected) algebraic normal $k$-subgroup of $\GG.$
   Then there exist    connected almost $k$-simple   normal   $k$-subgroups $\GG_1, \dots, \GG_r$ of $\GG$, a  subgroup $F$ of $\LL(k)$
 contained in the (finite)  center of  $\GG$, and 
    a connected normal $k$-subgroup $\HH$ of $\GG$ with the following properties:
    \begin{itemize}
    \item[(i)]  $\GG(k)$ is the almost direct product of $\LL(k)$ and $\HH(k)$;
   \item[(ii)]  $L(\kk)$ is the almost direct product of  $F, \GG_1(k), \dots, \GG_r(k)$; 
  \item[(iii)]  every $\GG_i(k)$ is  generated by its unipotent one parameter subgroups. 
  \end{itemize}
  \end{proposition}
  \begin{proof}
  Let $\LL_0$ be the connected component of $\LL.$
  Let  $\GG_1, \dots, \GG_r$  be the connected almost $k$-simple   normal   $k$-subgroups of $\GG$ contained in $\LL_0.$ Then
    $\LL_0$ is the almost direct product of the $\GG_i$'s and there exists
    a  connected normal $k$-subgroup $\HH$ of $\GG$ such that 
    $\GG$ is the almost direct product of $\LL_0$ and $\HH$
    (see \cite[2.15]{Borel-Tits1}). 
    It follows that $\LL$ is the almost direct product of $\LL_0$ and $\LL\cap \HH$
    and hence that $\LL\cap \HH$ is finite, since $\LL_0$ has finite index in $\LL.$
   This implies that  $\LL\cap \HH$ is contained in the center of $\GG,$ as
   $\LL\cap \HH$ is a normal subgroup of $\GG.$
   
  Since $\GG(k)$ is generated by its unipotent one parameter subgroups,
  the same is true for $\GG_i(k)$ for every $i$, for $\LL_0(k)$, 
  and for $\HH(k).$ Hence, 
  $\LL_0(k)$  is the almost direct product of   $\GG_1(k), \dots, \GG_r(k),$ 
 and  $\GG(k)$ is the almost direct product of $\LL_0(k)$ and $\HH(k)$
  (see \cite[Proposition 6.2]{Borel-Tits2}). 
  It  follows that
  $\LL(k)$ is the almost direct product of $\LL_0(k)$ and 
  $F:=\LL(k)\cap \HH(k)$. 
 From what we have seen above, $F\subset \LL\cap \HH$ is a subgroup of the center of 
 $\GG.$
  \end{proof}

 \section{A few examples}
  \label{S:Examples}
 \subsection{Abelian unipotent radical}
 \label{SS: AbelianRadical}
 Let $\LL$ be a  quasi-simple algebraic group  defined 
 over $\QQ.$ Assume that $L=\LL(\QQ)$ is generated by its
 unipotent one-parameter subgroups.
 Let $\LL\to GL(V)$ be a finite dimensional representation defined over $\QQ$
 of dimension at least 2; assume that $L\to GL(V(\QQ))$ is irreducible.
 Set $G=L\ltimes V(\QQ).$ Then $G$ is the group of
$\QQ$-rational points of the algebraic group $\LL\ltimes V$.
Denote by $F$ the kernel  of $L\to GL(V(\QQ))$
and observe that $F$ is a subgroup of the finite center of $\LL.$ We claim that
$$\Char(G)= \Char(L) \cup \{\widetilde{\chi} \mid \chi \in \Char(L)\}.$$
Indeed, let $\la\in \widehat{V(\QQ)}.$ The sets 
$K_\la$ and $P_\la$ as in Theorem~\ref{Theo-GenAlgGroup}
are $L$-invariant linear subspaces of $V(\QQ)$ and so
are equal either to $V(\QQ)$ or to $\{0\},$
by irreducibility of the representation of $L$ on $V(\QQ).$

\begin{itemize}
\item Assume that $K_\la= V(\QQ);$ then $\la=  \Un_{V(\QQ)}$ and 
$L_\la= L$. So, the characters of $G$ associated to $\la$ are
the characters of $L$ lifted to $G.$
\item  Assume that $K_\la=\{0\}$. Then $P_\la=\{0\}$ (see Proposition~\ref{Pro-KCenter}).
So,  $L_\la=F$ and every element of $\Char(G)$ associated to $\la$
is of the form $\widetilde{\chi}$ for some $\chi \in \Char(F).$
\end{itemize}

For instance, for every faithful $\QQ$-irreducible rational representation $\mathbf{SL}_2 \to GL(V),$ we have
$$
\Char(SL_2(\QQ)\ltimes V(\QQ))= \{\Un_G, \eps, \delta_e\},
$$
where $\eps$ is defined by 
$\eps(I, v)=1, \, \eps(-I, v)=-1,$ and $\eps(g, v)=0$ otherwise.

\subsection{The Heisenberg group as unipotent radical}
\label{SS: Heisenberg}
For an integer $n\geq 1$, consider the symplectic form $\beta$ on $\CC^{2n}$ given by
$$\beta((x,y),(x',y'))= (x,y)^t J (x',y')\tout (x,y),(x',y')\in \CC^{2n},$$
where $J$ is the $(2n\times 2n)$-matrix 
\[
J=\left(
\begin{array}{cc}
 0&I\\
-I_{n}& 0
\end{array}\
\right).
\]
The symplectic group 
$$
Sp_{2n}= \{g\in GL_{2n}(\CC)\mid {^{t}g}Jg=J\}
$$
is an algebraic which is a quasi-simple and defined over $\QQ.$

The $(2n+1)$-dimensional Heisenberg groups is the unipotent algebraic
 group $H_{2n+1}$ defined over $\QQ,$ with underlying set $\CC^{2n}\times \CC$ and product
$$((x,y),s)((x',y'),t)=\left((x+x',y+y'),s + t + \dfrac{1}{2}\beta((x,y),(x',y'))\right),
$$
for $(x,y), (x',y')\in \CC^{2n},\, s,t\in \CC.$ 

The group $Sp_{2n}$  acts by rational automorphisms of $H_{2n+1},$ given by 
$$
 g((x,y),t)= (g(x,y),t) \tout g\in Sp_{2n}(\QQ),\ (x,y)\in \QQ^{2n},\ t\in \QQ.
$$
Let 
$$G= Sp_{2n}(\QQ)\ltimes H_{2n+1}(\QQ)$$ 
be  the group of  $\QQ$-points of the algebraic group $Sp_{2n} \ltimes H_{2n+1}$
defined over $\QQ.$ Since $Sp_{2n}$ is $\QQ$-split,
$G$ is generated by its unipotent one-parameter subgroups.
We claim that 
$$
\Char(G)= \{\Un_G, \Un_{H_{2n+1}(\QQ)}, \eps\} \cup \{ \widetilde{\chi}\mid \chi\in \widehat{Z}\}
$$
where   $Z=\{((0,0),s)\ :\ s\in \QQ\}$ is the center of $H_{2n+1}(\QQ)$
and   $\eps$ is the character of $G$ defined by 
$$\eps(I, h)=1, \, \eps(-I, h)=-1, \, \text{ and } \, \eps(g, h)=0$$ 
for $g\in Sp_{2n}(\QQ)\setminus \{\pm I \}$ and $h\in H_{2n+1}(\QQ).$
 
Indeed, the Lie algebra of $H_{2n+1}(\QQ)$ is the $2n+1$-dimensional nilpotent Lie algebra $\mathfrak{h}$ over $\QQ$ with underlying set $\QQ^{2n}\times \QQ$ and 
Lie bracket
$$
[(x,y),s), ((x',y'),t] =\left(0, \beta((x,y),(x',y'))\right),
$$
for $(x,y), (x',y')\in \QQ^{2n},\, s,t\in \QQ.$
The action of $Sp_{2n}(\QQ)$ on $\mathfrak{h}$ is given by the same formula as for
the action on $H_{2n+1}(\QQ)$.

The $Sp_{2n}(\QQ)$-invariant ideals  $\mathfrak{k}$ of  $\mathfrak{h}$  are $\{0\},\mathfrak{h},$ and the center  $\mathfrak{z}$ of  $\mathfrak{h}.$
The corresponding ideals $\mathfrak{p}$, which are  inverse images in $\mathfrak{h}$
of the $G$-fixed elements in $\mathfrak{h}/\mathfrak{k}$, are respectively
$\mathfrak{z},\mathfrak{h}$ and $\mathfrak{z}.$

Let $\la\in \widehat{\mathfrak{h}}.$ 
\begin{itemize}
\item  Assume that $\mathfrak{k}_\la=\{0\}$. 
Then
$\mathfrak{p}_\la=\mathfrak{z}.$ Since 
 no element in $Sp_{2n}(\QQ)\setminus\{e\}$ acts trivially on  $\mathfrak{h}/ \mathfrak{z}\cong \QQ^{2n}$, we have $L_\la=\{e\}$. So, the only character of $G$ associated to $\la$  is $\widetilde{\chi_\la}.$ (Observe that $\widetilde{\chi_\la}\neq \Un_Z,$ since $\mathfrak{k}_\la=\{0\}.$) 

\item Assume that $\mathfrak{k}_\la= \mathfrak{h};$ then $\la=\Un_{\mathfrak{h}}$ and
$L_\la=Sp_{2n}(\QQ) $. So, the characters of $G$ associated to $\la$ are
the characters of $Sp_{2n}(\QQ)$ lifted to $G,$ 
that is, $\Un_G$, $\Un_{H_{2n+1}(\QQ)},$ and $\eps.$
 \item  Assume that $\mathfrak{k}_\la=\mathfrak{z}$. Then  $\mathfrak{p}_\la=\mathfrak{z}$ and $L_\la=\{e\}.$ So, $\Un_Z$ is the only  character 
 of $G$ associated to $\la$.
\end{itemize}

\subsection{Free nilpotent groups as unipotent radical}
Let ${\mathfrak u}={\mathfrak u}_{n,2} $ be  the free 2-step nilpotent Lie algebra on $n\geq 2$ generators
over $\QQ$;  as is well-known (see \cite{Gauger}), ${\mathfrak u}$   
 can be realized as follows. 
 Let $V$ be an $n$-dimensional vector space over $\QQ$ and set 
$ {\mathfrak u}:=  V\otimes \wedge^2 V,$ where $\wedge^2 V$ is the second exterior power 
of $V.$ The  Lie bracket  on  ${\mathfrak u}= V\otimes \wedge^2 V$ is defined by 
 $$
[(v_1, w_1) , (v_2, w_2)] = (0,  v_1\wedge v_2) \tous v_1, v_2 \in V, w_1, w_2\in \wedge^2 V.
$$
 The center of $\mathfrak u$ is $ \wedge^2 V$  and the associated unipotent  group 
 $U$ is $V \oplus \wedge^2 V$ with the product
 $$
 (v_1,w_1)(v_2,w_2)= (v_1+v_2, w_1+w_2+\frac{1}{2} v_1\wedge v_2) $$
 so that the exponential mapping $\exp: {\mathfrak u}\to U$ is the identity.
    The group $GL_n$ acts naturally on $V$ as well as on $\wedge^2V$
 and these actions induce  an action of $GL_n$ by automorphisms on $U$ given by 
  $$
  g(v, w)= (gv, gw) \tout g\in GL_n(\QQ),  v\in V, w\in \wedge^2V.
  $$
  
 Since $U$ coincides with the Heisenberg group $H_{3}(\QQ)$
 when $n=2,$ we may assume that $n\geq 3.$ 
  Let $L$ be  the group of $\QQ$-points  of an algebraic subgroup
  of $GL_{n}$ defined and quasi-simple over $\QQ.$
  Assume that the representations of $L$ on $V$ and on $\wedge^2 V$ are faithful and irreducible over $\QQ$
  and that, moreover,  $L$ is generated by its unipotent one-parameter subgroups
  (an example of such a group is $L=SL_n(\QQ)$).
  The group $G= L\ltimes U$ is  the group of  $\QQ$-points of an algebraic group
defined over $\QQ$ and $G$ is generated by its unipotent one-parameter subgroups.

We claim that 
$$
\Char(G)= \{\Un_G\} \cup \{\widetilde{\chi}\circ p\mid \chi\in \widehat{Z(L)}\} \cup \{\delta_e\} \cup \{\Un_{\wedge^2V}\},
$$
where $p: G\to L$ is the canonical epimorphism and $Z(L)$ the center of $L.$

Indeed, the $L$-invariant ideals  $\mathfrak{k}$ of  $\mathfrak{u}$  are 
$\{0\},\mathfrak{u},$ and the center  $\mathfrak{z}=\wedge^2V.$
By irreducibility of 
the $L$-action on $L$ on $V$ and $\wedge^2V$,  the corresponding ideals $\mathfrak{p}$, which are  inverse images in $\mathfrak{u}$
of the $G$-fixed elements in $\mathfrak{u}/\mathfrak{k}$, are respectively
$\mathfrak{0},\mathfrak{u}$ and $\mathfrak{z}.$

Let $\la\in \widehat{\mathfrak{u}}.$ 
\begin{itemize}
\item  Assume that $\mathfrak{k}_\la=\{0\}$. 
Then
$\mathfrak{p}_\la=\{0\}$ and 
$L_\la=\{e\}$. So, the only character of $G$ associated to $\la$  is $\delta_e.$
\item Assume that $\mathfrak{k}_\la= \mathfrak{u}.$ 
Then the characters of $G$ associated to $\la$ are
the characters of $L$ lifted to $G,$ 
that is, $\{\Un_G\}$ and $\widetilde{\chi}\circ p$ for  $\chi\in \widehat{Z(L)}.$

 \item  Assume that $\mathfrak{k}_\la=\mathfrak{z}$. Then 
 $\mathfrak{p}_\la=\mathfrak{z}$ and 
$L_\la=\{e\}$. So, $\Un_{\wedge^2 V}$ is the only  character
 of $G$ associated to $\la$.
\end{itemize}

\end{document}